\documentclass[10pt,a4paper]{article}

\usepackage[T1]{fontenc}
\usepackage{lmodern}
\usepackage[british]{babel}
\usepackage{amsmath,amssymb,amsthm,mathtools}
\usepackage[a4paper,margin=3cm]{geometry}
\usepackage{parskip}
\usepackage{tikz}
\usepackage{enumitem}
\usepackage{microtype}
\usepackage{needspace}
\usepackage[hidelinks]{hyperref}
\hypersetup{
    pdftitle={Convex Order Beyond Dimension One: Projection Tests, Counterexamples and Gaussian Mixtures},
    pdfauthor={Olivier Guéant}
}
\newcommand{\R}{\mathbb{R}}
\newcommand{\sphere}[1]{\mathbb{S}^{#1}}
\newcommand{\E}{\mathbb{E}}
\newcommand{\distribution}{\mathcal{L}}
\newcommand{\Gauss}{\mathcal{N}}
\newcommand{\cxle}{\lesssim_{\mathrm{cx}}}
\newcommand{\icxle}{\lesssim_{\mathrm{icx}}}
\newcommand{\projcxle}{\lesssim_{\mathrm{1d\text{-}proj\text{-}cx}}}

\theoremstyle{definition}
\newtheorem{definition}{Definition}
\theoremstyle{remark}
\newtheorem{remark}{Remark}
\theoremstyle{plain}
\newtheorem{lemma}{Lemma}
\newtheorem{proposition}{Proposition}
\newtheorem{theorem}{Theorem}

\title{Convex Order Beyond Dimension One:\\[0.3ex]
Projection Tests, Counterexamples and Gaussian Mixtures}
\author{Olivier Guéant\thanks{Université Paris Cité,
Laboratoire de Probabilités, Statistique et Modélisation,
75205 Paris Cedex 13, France.
Email: \href{mailto:olivier.gueant@u-pariscite.fr}
{\texttt{olivier.gueant@u-pariscite.fr}}.}}
\date{}

\begin{document}

\maketitle

\begin{abstract}
In actuarial science and quantitative finance, convex order provides a
natural way to compare risks with the same mean. In dimension one,
convex order is well understood through several characterisations. In higher dimensions, a natural approach is to compare
all one-dimensional projections, but, although necessary, the resulting condition
is in general not sufficient for multivariate convex order. A simple counterexample due to Pinelis exploits Popoviciu's inequality. We
revisit this example using gauge functions, and show that gauge
functions themselves do not suffice to characterise multivariate convex order
for symmetric distributions, even in dimension two. We
then show that the radial case makes the source of the failure
transparent and construct radial counterexamples in every dimension at least two. Nevertheless, projection tests are sufficient for some important
families of distributions. In particular, this is the case for linear transformations of a common radially distributed random vector. This framework
includes elliptical distributions generated from a common radial distribution
and, in particular, centred Gaussian distributions. Beyond comparisons between Gaussian distributions, Jourdain and
Pag\`es recently studied the comparison of a Gaussian distribution with
a Gaussian mixture, obtaining necessary conditions and sufficient conditions whose
equivalence was left open in general. In this paper, we settle the comparison of a centred Gaussian
distribution with a mixture of two centred Gaussian distributions:
in every finite dimension, the Gaussian is dominated by the mixture
in multivariate convex order if and only if all one-dimensional
projections are ordered in convex order.
\end{abstract}

\section{Introduction}

\subsection{Convex order in dimension one}

Comparing uncertain profits and losses is a basic problem in actuarial science and quantitative finance. For real-valued random variables \(X\) and \(Y\) in \(L^1\), a natural comparison is the increasing convex order:\footnote{Throughout the paper, expectations of convex functions of \(L^1\) random variables are understood in the extended sense and may equal \(+\infty\). These integrals are well defined because an affine lower bound ensures that the negative part belongs to \(L^1\).}\textsuperscript{,}\footnote{The same order is obtained if one replaces nondecreasing functions by 
increasing functions in the definition, hence the name. Indeed, for any nondecreasing convex function \(\phi\),
the function \(x\mapsto\phi(x)+\varepsilon x\) is
increasing for all \(\varepsilon>0\), and one can let
\(\varepsilon\downarrow0\).}\textsuperscript{,}\footnote{When \(X\) and \(Y\) represent losses, this means that any agent with increasing concave utility of wealth weakly prefers \(X\)~to~\(Y\), at a common initial wealth.}
\begin{equation*}
\begin{aligned}
    X\icxle Y\quad\Longleftrightarrow\quad
    &\E[\phi(X)]\le\E[\phi(Y)]\\
    &\text{for all nondecreasing convex functions }\phi:\R\to\R.
\end{aligned}
\end{equation*}

Increasing convex order accounts for both the level and the variability of losses. In actuarial terminology, it is the stop-loss order. Indeed, writing
\[
    C_X: t\in\R\longmapsto \E[(X-t)_+]
           =\int_t^{+\infty} \overline F_X(s)\,ds,
\]
where \(\overline F_X=1-F_X\) is the survival function of \(X\), we have
\begin{equation}\label{eq:stop-loss}
    X\icxle Y
    \quad\Longleftrightarrow\quad
    \forall t\in\R,\ C_X(t)\le C_Y(t).
\end{equation}
For a nonnegative random variable \(X\) representing insurance losses and a retention \(t\ge0\), \(C_X(t)\) is the pure premium of a stop-loss contract. This characterisation explains the role of increasing convex order in the comparison of insurance portfolios and reinsurance arrangements; see~\cite{KaasEtAl2008,ShakedShanthikumar2007}.

Increasing convex order also admits a characterisation based on quantiles:
\[
    X\icxle Y
    \quad\Longleftrightarrow\quad \forall u\in(0,1),\
    \int_u^1 F_X^{-1}(s)\,ds
    \le
    \int_u^1 F_Y^{-1}(s)\,ds,
\]
where \(F_X^{-1}(u)=\inf\{x\in\R\mid F_X(x)\ge u\}\) is the quantile function.\footnote{Equivalently, the expected shortfalls
\(\operatorname{ES}_u(X)=\frac{1}{1-u}\int_u^1F_X^{-1}(s)\,ds\) and \(\operatorname{ES}_u(Y)=\frac{1}{1-u}\int_u^1F_Y^{-1}(s)\,ds\)
are ordered at every confidence level \(u\in(0,1)\).}

Thus, increasing convex order can be characterised through explicit inequalities involving survival functions or quantile functions.

Convex order is the equal-mean case of this comparison. For real-valued random variables \(X\) and \(Y\) in \(L^1\),
\begin{equation*}
\begin{aligned}
    X\cxle Y
    \quad&\Longleftrightarrow\quad
    \E[\phi(X)]\le\E[\phi(Y)]\\[-0.3ex]
    &\hspace{4em}\text{for all convex functions }\phi:\R\to\R\\
    &\Longleftrightarrow\quad
    X\icxle Y\quad\text{and}\quad\E[X]=\E[Y].
\end{aligned}
\end{equation*}

\begin{remark}[Stop-loss, put, and absolute-deviation test functions]
\label{rem:scalar-tests}
For real-valued random variables \(X\) and \(Y\) in \(L^1\), the following conditions are equivalent:\footnote{In condition (iv), equality of means is obtained by letting \(t\to\pm\infty\) after subtracting \(|t|\).}
\begin{enumerate}[label=(\roman*)]
    \item \(X\cxle Y\);
    \item \(\E[X]=\E[Y]\) and \(\forall t\in\R,\ \E[(X-t)_+]\le\E[(Y-t)_+]\);
    \item \(\E[X]=\E[Y]\) and \(\forall t\in\R,\ \E[(X-t)_-]\le\E[(Y-t)_-]\);
    \item \(\forall t\in\R,\ \E[|X-t|]\le\E[|Y-t|]\).
\end{enumerate}
\end{remark}

Convex order therefore compares dispersion at a fixed mean. Its quantile characterisation is\footnote{This is the continuous counterpart of the partial-sum criterion for majorisation: for uniform distributions on finite sets of a given cardinality, it reduces to comparison of the partial sums of the support points arranged in increasing order, with equality of the total sums. See~\cite{HardyLittlewoodPolya1952} for the classical theory of majorisation and its applications.}
\[
    X\cxle Y
    \quad\Longleftrightarrow\quad
    \begin{cases}
    \displaystyle\int_0^uF_X^{-1}(s)\,ds
       \ge\int_0^uF_Y^{-1}(s)\,ds,
       &\forall u\in(0,1),\\[1ex]
    \displaystyle\int_0^1F_X^{-1}(s)\,ds
       =\int_0^1F_Y^{-1}(s)\,ds.
    \end{cases}
\]

Beyond analytic characterisations, Strassen's theorem provides a probabilistic characterisation of convex order -- see~\cite{Strassen1965}:
\[
    X\cxle Y
    \quad\Longleftrightarrow\quad
    \begin{array}{c}
    \text{there exist jointly defined }\widetilde X\stackrel{\mathcal D}{=}X
    \text{ and }\widetilde Y\stackrel{\mathcal D}{=}Y\\[0.2ex]
    \text{such that }\E[\widetilde Y\mid\widetilde X]=\widetilde X,
    \end{array}
\]
where \(\stackrel{\mathcal D}{=}\) denotes equality in distribution. The joint distribution of \((\widetilde X,\widetilde Y)\) is called a martingale coupling. Equivalently, \(\widetilde Y=\widetilde X+\varepsilon\), where \(\E[\varepsilon\mid\widetilde X]=0\). This is the mean-preserving-spread interpretation of increasing risk developed in economics by Rothschild and Stiglitz~\cite{RothschildStiglitz1970}.

\subsection{Multivariate convex order}

The definition of convex order extends directly to \(\R^d\)-valued random variables \(X\) and \(Y\) in \(L^1\):
\[
    X\cxle Y
    \quad\Longleftrightarrow\quad
    \E[\phi(X)]\le\E[\phi(Y)]
    \quad\text{for all convex functions }\phi:\R^d\to\R.
\]
Strassen's martingale-coupling characterisation remains valid in every dimension, giving the same interpretation in terms of conditionally mean-zero noise. This viewpoint underlies martingale optimal transport and its applications to robust pricing and hedging in mathematical finance; see, for example, \cite{BeiglbockHenryLaborderePenkner2013}.

For finitely supported distributions, Strassen's theorem yields a linear programming formulation. Suppose that \(X\) and \(Y\) have respective supports \(\{x_1,\ldots,x_m\}\) and \(\{y_1,\ldots,y_n\}\), with
\[
    \forall i\in\{1,\ldots,m\},\ \mathbb P(X=x_i)=p_i \quad\text{and}\quad \forall j\in\{1,\ldots,n\},\ \mathbb P(Y=y_j)=q_j.
\]
Then \(X\cxle Y\) if and only if there exist nonnegative coupling
masses \((\pi_{ij})_{1\le i\le m,\,1\le j\le n}\) satisfying
\[
\begin{cases}
    \displaystyle\sum_{j=1}^n\pi_{ij}=p_i,
    &\forall i\in\{1,\ldots,m\},\\[1ex]
    \displaystyle\sum_{i=1}^m\pi_{ij}=q_j,
    &\forall j\in\{1,\ldots,n\},\\[1ex]
    \displaystyle\sum_{j=1}^n\pi_{ij}y_j=p_i x_i,
    &\forall i\in\{1,\ldots,m\}.
\end{cases}
\]
After writing each vector equation of the last line coordinate by coordinate, this is a feasibility problem with \(mn\) nonnegative variables and \(m+n+md\) scalar equality constraints, some of which are redundant. It can be submitted to a standard linear programming solver with a constant objective. Thus, in the case of finite-support distributions there exists a numerical route in any dimension, although the support sizes and numerical tolerances determine its practical scope.

While this linear programming formulation provides a practical criterion for finitely supported distributions, finding tractable analytic criteria for general distributions beyond dimension one, comparable to the scalar stop-loss and quantile criteria, remains a challenge. A natural approach is to reduce the comparison to one-dimensional projections, to which these scalar criteria apply. When the coordinates represent asset losses or risk factors, this corresponds to comparing all linear portfolios.

\subsection{One-dimensional projected convex order}

The above discussion motivates the following definition:

\begin{definition}[One-dimensional projected convex order]
For \(\R^d\)-valued random variables \(X\) and \(Y\) in \(L^1\), we write
\begin{equation}\label{eq:projected-cx}
\begin{aligned}
    X\projcxle Y
    &\quad\Longleftrightarrow\quad
    \forall \xi\in\R^d,\ \xi'X\cxle\xi'Y\\
    &\quad\Longleftrightarrow\quad
    \forall \xi\in\sphere{d-1},\ \xi'X\cxle\xi'Y.
\end{aligned}
\end{equation}
\end{definition}

\begin{remark}
By the scalar stop-loss characterisation \eqref{eq:stop-loss},
\eqref{eq:projected-cx} is equivalent to
\[
    \forall (\xi,t)\in\sphere{d-1}\times\R,\
    \E[(\xi'X-t)_+]\le\E[(\xi'Y-t)_+].
\]
Indeed, these inequalities imply \(\xi'X\icxle\xi'Y\) for both
\(\xi\) and \(-\xi\), hence equality of the projected means.
\end{remark}

Multivariate convex ordering implies one-dimensional projected convex
ordering. Since a probability distribution on \(\R^d\) is determined
by the distribution of its one-dimensional projections, by injectivity of
the Fourier transform, it is natural to ask whether ordering
all projections characterises multivariate convex order. It turns out that the answer
is negative in every dimension \(d\ge2\).

In Theorem~\ref{thm:pinelis}, we revisit the geometric counterexample
of Pinelis~\cite{Pinelis2026}, using the gauge of the
triangle as the decisive test function. The natural role played by the gauge in that counterexample motivates Proposition~\ref{prop:homogeneous-planar}, which shows that, for symmetric distributions\footnote{We say that an \(\R^d\)-valued random variable \(X\) has symmetric distribution when \(X\stackrel{\mathcal D}{=}-X\).} in dimension \(d=2\), one-dimensional projected convex order implies comparison
for all convex functions that are positively homogeneous of degree \(1\). Consequently, a gauge function cannot play the same decisive role for symmetric distributions as it does in Pinelis's counterexample.

This observation might suggest that symmetry restores the equivalence
between one-dimensional projected and multivariate convex order in
dimension \(d=2\). However, Theorem~\ref{thm:radial} rules out this
possibility: in every dimension \(d\ge2\), there exist radial,
hence symmetric, distributions that are ordered in one-dimensional
projected convex order but not in multivariate convex order.
Beyond providing counterexamples, the radial setting clarifies
the distinction between the two orders: multivariate convex order
is equivalent to increasing convex order between the radii,
whereas one-dimensional projected convex order is equivalent to
increasing convex order between the radii after multiplication by
the same nonnegative angular factor, independent of the radii
(see Propositions~\ref{prop:radial-cx}
and~\ref{prop:radial-projections}).

Despite these general negative results, one-dimensional projection comparisons do characterise multivariate convex order within some important families of probability distributions.

Let \(Z\) be an \(\R^k\)-valued random variable in \(L^1\) with a radial distribution, meaning that \(OZ\stackrel{\mathcal D}{=}Z\) for all \(O \in \mathcal O_k\). Jourdain and Pag\`es~\cite{JourdainPages2022} proved that, for  all  matrices \(A,B\in \mathcal M_{d,k}\),
\[
    AA'\le BB'\quad\Longrightarrow\quad AZ\cxle BZ,
\]
where \(\le\) denotes here the classical order on symmetric matrices. If \(\mathbb P(Z\ne0)>0\), the converse holds even under the weaker assumption \(AZ\projcxle BZ\). Consequently, whenever the radial random variable \(Z\) is not the constant \(0\), we have
\[
    AZ\cxle BZ
    \quad\Longleftrightarrow\quad
    AZ\projcxle BZ
    \quad\Longleftrightarrow\quad
    AA'\le BB'.
\]

\begin{remark}[The converse under an \(L^1\) assumption]
Under an \(L^2\) assumption, the converse follows by testing the functions \(x\mapsto(\xi'x)^2\). Under the sole assumption \(Z=(Z_1,\ldots,Z_k)'\in L^1\), absolute-value tests suffice. Indeed, rotational invariance gives\footnote{If \(A'\xi\ne0\), consider \(O_A \in \mathcal O_k\) whose first column is
the unit vector \(A'\xi/\|A'\xi\|\). Then
\(O_A'A'\xi=\|A'\xi\|e_1\), so
\(\xi'AZ\stackrel{\mathcal D}{=}\xi'AO_AZ
=\|A'\xi\|Z_1\).
If \(A'\xi=0\), both sides of the identity are zero.}
\[
    \forall \xi\in\R^d,\ \E[|\xi'AZ|]
    =\sqrt{\xi'AA'\xi}\,\E[|Z_1|],
\]
and similarly for \(B\). Since \(\mathbb P(Z\ne0)>0\), rotational invariance ensures
\(\E[|Z_1|]>0\).\footnote{If \(\E[|Z_1|]=0\), then \(Z_1=0\)
almost surely. Rotational invariance implies that every coordinate
has the same distribution as \(Z_1\), so \(Z=0\) almost surely,
a contradiction.}
Testing \(AZ\projcxle BZ\) with the absolute-value function and
dividing by \(\E[|Z_1|]\), which is finite since \(Z\in L^1\), therefore yields \(AA'\le BB'\).
\end{remark}

\begin{remark}
The above equivalence applies to elliptical distributions obtained
as linear images of a common radial random variable. The requirement
of a common radial distribution is essential -- see Theorem~\ref{thm:radial}.
\end{remark}

In particular, centred Gaussian vectors admit such a representation
with a common standard Gaussian vector. Thus, for
\(X\sim\Gauss(0,\Sigma_X)\) and \(Y\sim\Gauss(0,\Sigma_Y)\),
where \(\Sigma_X,\Sigma_Y\in \mathcal S_d^+\), the two orders coincide
and are characterised by the comparison of the variance-covariance matrices:
\begin{equation}\label{eq:gaussian-covariance}
    X\cxle Y
    \quad\Longleftrightarrow\quad
    \Sigma_X\le\Sigma_Y.
\end{equation}

\subsection{Comparing a Gaussian distribution and a Gaussian mixture
in both directions}

Jourdain and Pag\`es investigate in \cite{JourdainPages2025} convex
comparison between a Gaussian distribution and a Gaussian mixture.
Although their framework is more general, we focus in this paper on the specific case of mixtures of \(n=2\) centred Gaussian distributions. For that purpose, consider \(p\in(0,1)\) and
\(\Sigma,\Sigma_1,\Sigma_2\in \mathcal S_d^+\).

For
\[
    X\sim p\Gauss(0,\Sigma_1)+(1-p)\Gauss(0,\Sigma_2) \quad\text{and}\quad Y\sim\Gauss(0,\Sigma),
\]
their results give
\[
    X\cxle Y
    \quad\Longleftrightarrow\quad
    X\projcxle Y
    \quad\Longleftrightarrow\quad
    \Sigma_1\le\Sigma
    \ \text{and}\
    \Sigma_2\le\Sigma.
\]
Thus, multivariate and one-dimensional projected convex order
coincide in this setting.

The reverse comparison is more subtle. Consider
\[
    X\sim\Gauss(0,\Sigma)\quad\text{and}\quad
    Y\sim p\Gauss(0,\Sigma_1)+(1-p)\Gauss(0,\Sigma_2).
\]
Jourdain and
Pag\`es prove in \cite{JourdainPages2025} that
\begin{equation}\label{eq:necessary}
    X\projcxle Y
    \quad\Longleftrightarrow\quad \forall \xi\in\R^d,\
    \sqrt{\xi'\Sigma\xi}
    \le p\sqrt{\xi'\Sigma_1\xi}
         +(1-p)\sqrt{\xi'\Sigma_2\xi}.
\end{equation}

They also give a sufficient condition for multivariate convex order based on a joint Gaussian coupling. With \(P=(pI_d\ \ (1-p)I_d)\), this condition reads\footnote{We recall the argument, as the same coupling idea will play a central role in Section \ref{sec:gaussian-mixtures}. Let \((Z_1,Z_2)\sim\Gauss(0,\Gamma)\). Then \(Z_1\sim\Gauss(0,\Sigma_1)\) and \(Z_2\sim\Gauss(0,\Sigma_2)\), while
\[
    pZ_1+(1-p)Z_2\sim\Gauss(0,P\Gamma P').
\]
Under \eqref{eq:sufficient}, Gaussian comparison \eqref{eq:gaussian-covariance} therefore gives \(X\cxle pZ_1+(1-p)Z_2\). Hence, for every convex function \(\phi:\R^d\to\R\), Jensen's inequality yields
\[
    \E[\phi(X)]\le\E[\phi(pZ_1+(1-p)Z_2)]\le p\E[\phi(Z_1)]+(1-p)\E[\phi(Z_2)]=\E[\phi(Y)],
\]
that is, \(X\cxle Y\).}
\begin{equation}\label{eq:sufficient}
    \exists\,C\in\mathcal M_d \quad\text{such that}\quad \Gamma=\begin{pmatrix}\Sigma_1&C\\C'&\Sigma_2\end{pmatrix}\in \mathcal S_{2d}^+ \quad\text{and}\quad \Sigma\le P\Gamma P'.
\end{equation}

Consequently, \eqref{eq:sufficient} implies multivariate convex
order, which in turn implies one-dimensional projected convex
order and thus the inequalities in~\eqref{eq:necessary}.
Jourdain and Pag\`es establish equivalence between these conditions
under several additional matrix assumptions. In the present paper,
we prove that no such assumptions are needed for mixtures of \(n=2\) centred Gaussian distributions, in any dimension \(d\ge 2\).

We first establish the result in dimension \(d=2\) (Theorem~\ref{thm:dimension-two}), using a linear algebra lemma to construct a suitable Gaussian coupling. We then extend the argument by induction to arbitrary dimension \(d\ge 2\) (Theorem~\ref{thm:two-components-arbitrary-dimension}).

The remainder of the paper is organised as follows.
Section~\ref{sec:counterexamples} revisits Pinelis's counterexample,
examines convex test functions that are positively homogeneous of degree \(1\) for symmetric
distributions on \(\R^2\), and constructs radial counterexamples in every
dimension \(d\ge2\). Section~\ref{sec:gaussian-mixtures} proves
the Gaussian-mixture result, first in dimension \(d=2\) and then
in arbitrary dimension \(d\ge2\). Section~\ref{sec:conclusion} presents
some concluding remarks and paves the way for future research.

\section{Limitations of one-dimensional projection tests}\label{sec:counterexamples}
In this section, we first revisit Pinelis's triangle counterexample, which shows
that one-dimensional projected convex order is strictly weaker than multivariate
convex order in every dimension \(d\ge2\). We then study convex test functions that are positively
homogeneous of degree \(1\) for symmetric distributions in
dimension \(d=2\). Finally, we construct radial counterexamples in
every dimension \(d\ge2\): a first example obtained the old-fashioned
way, and a second example with an absolutely continuous distribution found using generative AI tools.

\subsection{A counterexample built on a triangle}

The fact that one-dimensional projected convex order is strictly weaker than multivariate convex order can be found in \cite{Pinelis2026}.\footnote{I first learned of this result through a MathOverflow post and subsequently discussed it with its author, Iosif Pinelis.} Below, we recast the triangle counterexample of \cite{Pinelis2026} and its application of Popoviciu's inequality. However, unlike in \cite{Pinelis2026}, the decisive convex test function used here is the gauge function of the triangle. To better understand the role played by the gauge function, let us first state a lemma that we call the boundary-support criterion.

\begin{lemma}[Boundary-support criterion]\label{lem:gauge-boundary}
Let \(\mathcal C\subset\R^d\) be open and convex, with \(0\in\mathcal C\). Let \(X\) and \(Y\) be \(\R^d\)-valued random variables in \(L^1\) satisfying
\[
    \mathbb P(X\in\partial\mathcal C)=1 \quad\text{and}\quad \mathbb P(Y\in\overline{\mathcal C})=1.
\]
If \(X\cxle Y\), then \(\mathbb P(Y\in\partial\mathcal C)=1\).
In particular, positive mass of \(Y\) in \(\mathcal C\)
rules out \(X\cxle Y\).
\end{lemma}

\begin{proof}
Define the gauge of \(\mathcal C\) by
\[
    j_{\mathcal C}: x\in\R^d
    \longmapsto
    \inf\{\lambda>0\mid x\in\lambda\mathcal C\}.
\]
Since \(\mathcal C\) is convex and contains a ball \(B(0,r)\)
for some \(r>0\), its gauge function is well defined, nonnegative,
convex, and satisfies
\[
    \forall x\in\R^d,\ 
    j_{\mathcal C}(x)\le\frac{\|x\|}{r}.
\]
Moreover,
\[
    \mathcal C=\{x\in\R^d\mid j_{\mathcal C}(x)<1\}
    \quad\text{and}\quad
    \overline{\mathcal C}
    =\{x\in\R^d\mid j_{\mathcal C}(x)\le1\},
\]
so \(j_{\mathcal C}=1\) on \(\partial\mathcal C\).
If \(X\cxle Y\), testing with \(j_{\mathcal C}\) gives
\[
    1=\E[j_{\mathcal C}(X)]
    \le\E[j_{\mathcal C}(Y)]\le1.
\]
Since \(0\le j_{\mathcal C}(Y)\le1\) almost surely, this implies
\(j_{\mathcal C}(Y)=1\) almost surely, hence
\(Y\in\partial\mathcal C\) almost surely.
\end{proof}

\begin{theorem}[One-dimensional projected convex order is strictly weaker]\label{thm:pinelis}
There exist finitely supported \(\R^2\)-valued random variables
\(X\) and \(Y\) such that
\[
    X\projcxle Y
    \quad\text{but}\quad
    X\not\cxle Y.
\]
\end{theorem}

\begin{proof}
Let \(A,B,C\in\R^2\) be three non-collinear points, and let
\(O\) be their centre of gravity, assumed without loss of generality to be the origin. Consider
the triangle
\[
    T=\operatorname{Conv}(\{A,B,C\}),
\]
and let
\[
    D=\frac{A+B}{2},\quad
    E=\frac{B+C}{2}
    \quad\text{and}\quad
    F=\frac{C+A}{2}
\]
be its three edge midpoints.

Let \(X\) be uniformly distributed on \(\{D,E,F\}\) and let
\(Y\) have distribution
\(
    \frac16\bigl(\delta_A+\delta_B+\delta_C\bigr)
    +\frac12\delta_O
\).
In other words, \(X\) takes each midpoint with probability \(1/3\),
while \(Y\) takes each vertex with probability \(1/6\) and
the centre of gravity with probability \(1/2\) -- see Figure~\ref{fig:triangle-counterexample}.

\begin{figure}[htbp]
\centering
\begin{tikzpicture}[scale=1.15]
    \coordinate (A) at (-2.3,-1.1);
    \coordinate (B) at (2.0,-0.7);
    \coordinate (C) at (0.3,1.8);
    \coordinate (O) at (0,0);
    \coordinate (D) at (-0.15,-0.9);
    \coordinate (E) at (1.15,0.55);
    \coordinate (F) at (-1.0,0.35);

    \draw[thick] (A)--(B)--(C)--cycle;

    \pgfmathsetmacro{\vertexradius}{2.3}
    \pgfmathsetmacro{\midpointradius}{2.3*sqrt(2)}
    \pgfmathsetmacro{\centreradius}{2.3*sqrt(3)}

    \foreach \P in {A,B,C}
        \fill (\P) circle[radius=\vertexradius pt];
    \fill (O) circle[radius=\centreradius pt];
    \foreach \P in {D,E,F}
        \filldraw[fill=white,draw=black,line width=0.9pt]
        (\P) circle[radius=\midpointradius pt];

    \node[anchor=north east,xshift=-3pt,yshift=-3pt]
        at (A) {\(A\)};
    \node[anchor=north west,xshift=3pt,yshift=-3pt]
        at (B) {\(B\)};
    \node[anchor=south,yshift=5pt] at (C) {\(C\)};
    \node[anchor=north west,xshift=5pt,yshift=-5pt]
        at (O) {\(O\)};
    \node[anchor=north,yshift=-5pt] at (D) {\(D\)};
    \node[anchor=west,xshift=5pt] at (E) {\(E\)};
    \node[anchor=east,xshift=-5pt] at (F) {\(F\)};
\end{tikzpicture}
\caption{The open circles represent the possible values of \(X\),
each with probability \(1/3\). The filled circles represent the
possible values of \(Y\): each vertex has probability \(1/6\),
and \(O\) has probability~\(1/2\). Marker areas are proportional
to the corresponding probabilities.}
\label{fig:triangle-counterexample}
\end{figure}
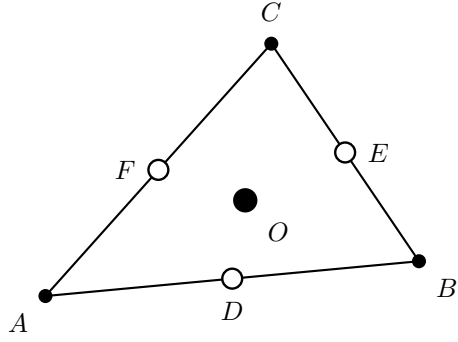

The distribution of \(X\) is supported on \(\partial T\),
whereas \(Y\) takes values in \(T\) and has probability \(1/2\)
of lying at its interior point \(O\).
Applying Lemma~\ref{lem:gauge-boundary} to the interior of \(T\)
therefore gives \(X\not\cxle Y\).

It remains to compare all one-dimensional projections.
Fix \(\xi\in\R^2\) and a convex function \(\phi:\R\to\R\),
and set
\[
    a=\xi'A,\quad b=\xi'B
    \quad\text{and}\quad c=\xi'C.
\]
Then
\[
\begin{aligned}
    \E[\phi(\xi'X)]
    &=\frac13\left(
        \phi\left(\frac{a+b}{2}\right)
       +\phi\left(\frac{b+c}{2}\right)
       +\phi\left(\frac{c+a}{2}\right)
      \right)\\
    \text{and}\quad
    \E[\phi(\xi'Y)]
    &=\frac16\bigl(\phi(a)+\phi(b)+\phi(c)\bigr)
      +\frac12\phi\left(\frac{a+b+c}{3}\right).
\end{aligned}
\]
Popoviciu's inequality~\cite{Popoviciu1965} states that,
for all convex functions \(\phi:\R\to\R\),
\[
    \forall a,b,c\in\R,\ 
    \frac{\phi(a)+\phi(b)+\phi(c)}{3}
    +\phi\left(\frac{a+b+c}{3}\right)
    \ge
    \frac23\left(
        \phi\left(\frac{a+b}{2}\right)
       +\phi\left(\frac{b+c}{2}\right)
       +\phi\left(\frac{c+a}{2}\right)
    \right).
\]
Dividing by \(2\) yields precisely
\[
    \E[\phi(\xi'X)]\le\E[\phi(\xi'Y)].
\]
Since \(\xi\) and \(\phi\) were arbitrary, \(X\projcxle Y\).
\end{proof}

\begin{remark}
Embedding the preceding pair in \(\R^d\) as
\[
    \bigl(X,\underbrace{0,\ldots,0}_{d-2}\bigr)
    \quad\text{and}\quad
    \bigl(Y,\underbrace{0,\ldots,0}_{d-2}\bigr)
\]
gives a counterexample in every dimension \(d\ge2\).
\end{remark}

\subsection{Symmetric distributions in dimension \texorpdfstring{\(d=2\)}{d=2} and the case of convex
test functions that are positively homogeneous of degree \texorpdfstring{\(1\)}{1}}

In the triangle counterexample above, the distributions are not
symmetric. Our interest in centred Gaussian
distributions and centred Gaussian mixtures leads us to focus
in this subsection on the case of symmetric distributions. For such distributions, every convex test function can be replaced by its even part without changing its expectation.

In dimension \(d=2\), even convex functions that are positively homogeneous of degree \(1\) admit a representation as integrals of absolute values of linear forms with respect to a nonnegative measure. This representation, recalled in Lemma~\ref{lem:angular-homogeneous}, shows that comparison under all these test functions already follows from one-dimensional projected convex order.

The proof of Lemma~\ref{lem:gauge-boundary} gives an intuitive
explanation of the role of gauges for convex ordering:
a gauge distinguishes mass on the boundary from mass in the
interior. The representation of Lemma \ref{lem:angular-homogeneous} below nevertheless shows that gauge functions cannot be decisive for finding symmetric planar distributions that are not ordered in multivariate convex order but
whose one-dimensional projections are in convex order.
During the research process that led to this paper, this observation raised the hope
that symmetry might restore equivalence between the two orders
in the specific case of dimension \(d=2\). However, this hope was dashed by the radial counterexamples of Section \ref{sec:radial-counterexamples}.

\begin{lemma}[Angular representation in the plane]
\label{lem:angular-homogeneous}
Let \(\phi:\R^2\to\R\) be even, convex, and positively
homogeneous of degree \(1\). For all \(\theta\in[0,2\pi)\), set
\[
    u_\theta=(\cos\theta,\sin\theta)',
    \quad
    v_\theta=(-\sin\theta,\cos\theta)'
    \quad\text{and}\quad
    h(\theta)=\phi(u_\theta),
\]
and extend these definitions periodically to \(\R\).
Then \(h\) is continuous, \(\pi\)-periodic, and satisfies
\[
    \forall r\ge0,\ \forall\theta\in\R,\quad
    \phi(ru_\theta)=rh(\theta).
\]
The distribution\footnote{Here the word distribution is in the sense of Schwartz, not in the sense of a probability distribution.} 
\(h+h''\) is a nonnegative, \(\pi\)-periodic measure \(\mu\), whose restriction to \([0,\pi)\) is finite, and
\begin{equation}\label{eq:homogeneous-absolute}
    \forall x\in\R^2,\ 
    \phi(x)=\frac12\int_{[0,\pi)}|v_\theta'x|\,\mu(d\theta).
\end{equation}
If \(\phi\) is of class \(C^2\) on \(\R^2\setminus\{0\}\),
then
\[
    \mu(d\theta)=\bigl(h(\theta)+h''(\theta)\bigr)\,d\theta,
\]
and
\begin{equation}\label{eq:angular-sine}
    \forall\theta\in\R,\ 
    h(\theta)=\frac12\int_\theta^{\theta+\pi}
       \sin(t-\theta)\bigl(h(t)+h''(t)\bigr)\,dt.
\end{equation}
\end{lemma}

\begin{proof}
Homogeneity and evenness of \(\phi\) give the polar representation and
\(\pi\)-periodicity of \(h\). Continuity of \(h\) follows from convexity, hence continuity, of \(\phi\).

Suppose first that \(\phi\) is of class \(C^2\) on \(\R^2\setminus\{0\}\).
Differentiating \(h(\theta)=\phi(u_\theta)\) twice gives
\[
    \forall\theta\in\R,\ 
    h''(\theta)
    =v_\theta'\nabla^2\phi(u_\theta)v_\theta
     -u_\theta'\nabla\phi(u_\theta).
\]
By Euler's identity for positively homogeneous functions
of degree \(1\),
\[
    \forall\theta\in\R,\ 
    u_\theta'\nabla\phi(u_\theta)=\phi(u_\theta)=h(\theta).
\]
Consequently, convexity of \(\phi\) yields
\[
    \forall\theta\in\R,\ 
    h''(\theta)+h(\theta)
    =v_\theta'\nabla^2\phi(u_\theta)v_\theta\ge0.
\]
Integrating twice by parts and using
\(h(\theta+\pi)=h(\theta)\), we obtain
\[
    \int_\theta^{\theta+\pi}\sin(t-\theta)h''(t)\,dt
    =2h(\theta)
     -\int_\theta^{\theta+\pi}\sin(t-\theta)h(t)\,dt.
\]
This proves \eqref{eq:angular-sine}. Since
\[
    \forall\theta,t\in\R,\ 
    |v_\theta'u_t|=|\sin(t-\theta)|,
\]
periodicity and homogeneity give
\eqref{eq:homogeneous-absolute} in the smooth case.

For the general case, consider a sequence of nonnegative \(C^\infty\) mollifier functions
\((\eta_m)_{m\ge1}\) on \(\R\) satisfying
\[
    \forall m \ge 1, \quad \operatorname{supp}(\eta_m)\subset[-1/m,1/m]
    \quad\text{and}\quad
    \int_\R\eta_m(\alpha)\,d\alpha=1.
\]

Writing \(\rho_\alpha\) for the rotation by angle \(\alpha\),
we define
\[
    \phi_m(x)
    =\int_\R\phi(\rho_\alpha(x))\eta_m(\alpha)\,d\alpha
    \quad\text{and}\quad
    h_m(\theta)=\phi_m(u_\theta).
\]
Each \(\phi_m\) is even, convex, and positively homogeneous
of degree \(1\). Moreover, a change of variable gives
\[
    \forall\theta\in\R,\ 
    h_m(\theta)
    =\int_\R h(\theta+\alpha)\eta_m(\alpha)\,d\alpha
    =\int_\R h(t)\eta_m(t-\theta)\,dt.
\]
The first expression shows that \(h_m\) is \(\pi\)-periodic.
The second shows that \(h_m\) is \(C^\infty\).
Since \(\phi_m(ru_\theta)=rh_m(\theta)\), it follows that
\(\phi_m\) is smooth on \(\R^2\setminus\{0\}\).

Uniform continuity of \(h\) and the support condition on
\(\eta_m\) give
\[
    \|h_m-h\|_\infty
    \le
    \sup_{\substack{\theta\in\R\\|\alpha|\le1/m}}
    |h(\theta+\alpha)-h(\theta)|
    \underset{m\to+\infty}{\longrightarrow}0.
\]
By homogeneity,
\[
    \forall x\in\R^2,\ 
    |\phi_m(x)-\phi(x)|
    \le\|h_m-h\|_\infty\|x\|.
\]

By the smooth case, the measures
\[
    \mu_m(d\theta)
    =\bigl(h_m(\theta)+h_m''(\theta)\bigr)\,d\theta
\]
are nonnegative and satisfy
\[
    \forall x\in\R^2,\ 
    \phi_m(x)
    =\frac12\int_0^\pi|v_\theta'x|\,\mu_m(d\theta).
\]
Their masses on \([0,\pi]\) satisfy
\[
    \mu_m([0,\pi])
    =\int_0^\pi h_m(\theta)\,d\theta,
\]
and are therefore bounded. Compactness of \([0,\pi]\)
allows us to extract a subsequence, still indexed by \(m\),
converging weakly to a nonnegative finite measure \(\nu\) on
\([0,\pi]\).

Since all integrands of interest are \(\pi\)-periodic,
we combine any masses at the two endpoints by setting
\[
    \mu=\nu\big|_{[0,\pi)}+\nu(\{\pi\})\delta_0.
\]
Thus, for all continuous \(\pi\)-periodic functions \(\psi\),
\[
    \lim_{m\to+\infty}\int_0^\pi\psi(\theta)\,\mu_m(d\theta)
    =\int_{[0,\pi)}\psi(\theta)\,\mu(d\theta).
\]
In particular, for all smooth \(\pi\)-periodic functions \(\psi\),
two integrations by parts, with boundary terms vanishing
by periodicity, give
\[
\begin{aligned}
    \int_{[0,\pi)}\psi(\theta)\,\mu(d\theta)
    &=\lim_{m\to+\infty}
      \int_0^\pi
      \psi(\theta)\bigl(h_m(\theta)+h_m''(\theta)\bigr)
      \,d\theta\\
    &=\lim_{m\to+\infty}
      \int_0^\pi
      h_m(\theta)\bigl(\psi(\theta)+\psi''(\theta)\bigr)
      \,d\theta\\
    &=\int_0^\pi
      h(\theta)\bigl(\psi(\theta)+\psi''(\theta)\bigr)
      \,d\theta.
\end{aligned}
\]
The periodic extension of \(\mu\) is therefore \(h+h''\)
in the distributional sense. Finally, for fixed \(x\),
the function \(\theta\mapsto|v_\theta'x|\) is continuous
and \(\pi\)-periodic. Passing to the limit gives
\[
    \phi(x)
    =\lim_{m\to+\infty}\phi_m(x)
    =\frac12\int_{[0,\pi)}|v_\theta'x|\,\mu(d\theta),
\]
which proves \eqref{eq:homogeneous-absolute}.
\end{proof}

\begin{remark}[Connection with convex geometry]
The preceding argument is classical in convex geometry (see \cite{Schneider2014}),
where it is applied to support functions. For a nonempty
origin-symmetric compact convex set \(K\subset\R^2\),
the support function
\[
    h_K(x)=\sup_{y\in K}x'y
\]
is even, convex, and positively homogeneous of degree \(1\).
Since the support function of a centred segment \([-v,v]\)
is \(x\mapsto|v'x|\), and support functions add under
Minkowski addition, \(K\) is a limit of finite Minkowski
sums of centred segments. This is the classical statement
in convex geometry that every (origin-)symmetric compact convex set
in \(\R^2\) is a zonoid.
\end{remark}

\begin{proposition}[Homogeneous test functions for symmetric planar distributions]
\label{prop:homogeneous-planar}
Let \(X\) and \(Y\) be \(\R^2\)-valued random variables in \(L^1\)
with symmetric distributions. Then
\[
    X\projcxle Y
    \quad\Longrightarrow\quad
    \E[\phi(X)]\le\E[\phi(Y)]
\]
for all even convex functions \(\phi:\R^2\to\R\) that are
positively homogeneous of degree \(1\). In particular,
\[
    X\projcxle Y
    \quad\Longrightarrow\quad
    \E[j_{\mathcal C}(X)]\le\E[j_{\mathcal C}(Y)]
\]
for all origin-symmetric open convex sets
\(\mathcal C\subset\R^2\) containing the origin.
\end{proposition}

\begin{proof}
Suppose that \(X\projcxle Y\), and let \(\phi\) be an even
convex function positively homogeneous of degree \(1\).
By Lemma~\ref{lem:angular-homogeneous},
\[
\begin{aligned}
    \E[\phi(X)]
    &=\frac12\int_{[0,\pi)}
      \E[|v_\theta'X|]\,\mu(d\theta)\\
    &\le\frac12\int_{[0,\pi)}
      \E[|v_\theta'Y|]\,\mu(d\theta)
     =\E[\phi(Y)].
\end{aligned}
\]
Finally, the gauge of an origin-symmetric open convex
set containing the origin is even, convex, and positively
homogeneous of degree \(1\).
\end{proof}

\begin{remark}
For symmetric distributions, the comparison also holds for
convex test functions that are positively homogeneous of degree
\(1\) but are not even: replacing \(\phi\) by its even part preserves these properties
and leaves its expectations unchanged. This includes gauges
of open convex sets containing the origin that are not
themselves symmetric.
\end{remark}

Although our initial motivation came from gauge functions,
the above angular representation yields a broader characterisation,
which does not require symmetry of the distributions.
By condition (iv) of Remark~\ref{rem:scalar-tests}, one-dimensional projected convex order is equivalent to
\[
    \forall (\xi,t)\in\sphere{1}\times\R,\ 
    \E[|\xi'X-t|]\le\E[|\xi'Y-t|].
\]
Combining these shifted absolute-value tests with
Lemma~\ref{lem:angular-homogeneous} gives comparison under
all translations of even convex
functions that are positively homogeneous of degree \(1\). Conversely, this class contains the shifted
absolute-value functions and therefore characterises one-dimensional projected
convex order.

\begin{proposition}[Translated homogeneous test functions]
\label{prop:translated-homogeneous}
Let \(X\) and \(Y\) be \(\R^2\)-valued random variables in \(L^1\). Then
\[
\begin{aligned}
    X\projcxle Y
    \quad\Longleftrightarrow\quad
    &\E[\phi(X-z)]\le\E[\phi(Y-z)]\\
    &\text{for all }z\in\R^2
      \text{ and all even convex functions }\phi:\R^2\to\R\\
    &\text{that are positively homogeneous of degree }1.
\end{aligned}
\]
\end{proposition}

\begin{proof}
Suppose that \(X\projcxle Y\). Fix \(z\in\R^2\) and an even
convex function \(\phi:\R^2\to\R\) that is positively homogeneous
of degree \(1\). Lemma~\ref{lem:angular-homogeneous} gives
\[
    \forall x\in\R^2,\ 
    \phi(x-z)
    =\frac12\int_{[0,\pi)}
      |v_\theta'x-v_\theta'z|\,\mu(d\theta).
\]
Since \(X\projcxle Y\), we have 
\[
\begin{aligned}
    \E[\phi(X-z)]
    &=\frac12\int_{[0,\pi)}
      \E[|v_\theta'X-v_\theta'z|]\,\mu(d\theta)\\
    &\le\frac12\int_{[0,\pi)}
      \E[|v_\theta'Y-v_\theta'z|]\,\mu(d\theta)
     =\E[\phi(Y-z)].
\end{aligned}
\]

Conversely, for \((\xi,t)\in\sphere{1}\times\R\), consider
\( \phi(x)=|\xi'x|\) and \(z=t\xi\)
to obtain
\(
    \E[|\xi'X-t|]\le\E[|\xi'Y-t|].
\)
We conclude that \(X\projcxle Y\).
\end{proof}

The class of test functions consisting of translations of
even convex functions positively homogeneous of degree \(1\)
is substantial: it includes translated gauges of all
origin-symmetric open convex sets containing the origin.
Since the sets \(\{x\in\R^2\mid\phi(x)\le c\}\) are symmetric
and convex whenever \(\phi\) is even and convex, one might
hope to extend the comparison to all even convex functions through a level-set argument. The next subsection
rules out this possibility: the two orders differ even
for radial distributions in dimension \(d=2\).

\subsection{Towards counterexamples with radial distributions}
\label{sec:radial-counterexamples}

We now construct radial counterexamples in every dimension \(d\ge2\). Their rotational symmetry reduces both orders to one-dimensional comparisons involving the radii, making the source of the failure particularly transparent. This leads to
two counterexamples of a rather different nature: the first, with two-point radius distributions, constructed
in the old-fashioned, artisanal way, and the second, with a density,
found with the help of generative AI tools.\footnote{The reader can judge
which counterexample they prefer. I personally prefer the first, but I am of course biased.}

\subsubsection{Multivariate convex order and one-dimensional projected convex order for radial distributions}

Throughout the following paragraphs, \(d\ge2\), \(U_d\) is uniformly distributed
on \(\sphere{d-1}\), and \(R\) and \(S\) are nonnegative random variables in \(L^1\), independent of \(U_d\). We consider \(\R^d\)-valued random variables \(X\) and \(Y\) such that
\[
    X\stackrel{\mathcal D}{=}RU_d
    \quad\text{and}\quad
    Y\stackrel{\mathcal D}{=}SU_d.
\]
We first characterise multivariate convex order in terms of the radii.

\begin{proposition}[Multivariate convex order for radial distributions]
\label{prop:radial-cx}
We have
\[
    X\cxle Y
    \quad\Longleftrightarrow\quad
    R\icxle S.
\]
\end{proposition}

\begin{proof}
Suppose first that \(X\cxle Y\). For all \(t\in\R\), the function
\(x\mapsto(\|x\|-t)_+\) is convex. Hence
\[
    \E[(R-t)_+]
    =\E[(\|X\|-t)_+]
    \le\E[(\|Y\|-t)_+]
    =\E[(S-t)_+].
\]
The stop-loss characterisation of increasing convex order gives
\(R\icxle S\).

Conversely, suppose that \(R\icxle S\), and let
\(\phi:\R^d\to\R\) be convex. Define
\[
    g: r \in \R \longmapsto \E[\phi(rU_d)].
\]
The function \(g\) is convex. Since \(U_d\) and \(-U_d\)
have the same distribution, it is also even, and therefore nondecreasing
on \(\R_+\). Thus \(\tilde g(r)=g(r_+)\) defines a nondecreasing and convex function on
\(\R\). By independence and the nonnegativity of the radii,
\[
    \E[\phi(X)]
    =\E[g(R)]
    =\E[\tilde g(R)]
    \le\E[\tilde g(S)]
    =\E[g(S)]
    =\E[\phi(Y)].
\]
This proves \(X\cxle Y\).
\end{proof}

To characterise one-dimensional projected convex order, we first recall an elementary
property of real-valued random variables with symmetric distributions.

\begin{lemma}[Convex order for real-valued random variables with symmetric distributions]
\label{lem:symmetric-cx}
Let \(\Xi\) and \(\Upsilon\) be two real-valued random variables in \(L^1\) with symmetric distributions. We have
\[
    \Xi\cxle \Upsilon
    \quad\Longleftrightarrow\quad
    |\Xi|\icxle|\Upsilon|.
\]
\end{lemma}

\begin{proof}
If \(\Xi\cxle \Upsilon\), testing with the convex functions
\(x\mapsto(|x|-t)_+\), for \(t\in\R\), gives
\(|\Xi|\icxle|\Upsilon|\).

Conversely, suppose that \(|\Xi|\icxle|\Upsilon|\), and let
\(\phi:\R\to\R\) be convex. Its even part
\[
    h : r \in \R \longmapsto \frac{\phi(r)+\phi(-r)}{2}
\]
is even and convex, so \(r\mapsto h(r_+)\) is nondecreasing and convex.
Because \(\Xi\) and \(\Upsilon\) have symmetric distributions, we obtain
\[
    \E[\phi(\Xi)]
    =\E[h(|\Xi|)]
    \le\E[h(|\Upsilon|)]
    =\E[\phi(\Upsilon)].
\]
Thus \(\Xi\cxle \Upsilon\).
\end{proof}

\begin{proposition}[One-dimensional projected convex order for radial distributions]
\label{prop:radial-projections}
Let \(V_d=e_1'U_d\) be the first coordinate of \(U_d\). Then
\[
    X\projcxle Y
    \quad\Longleftrightarrow\quad
    RV_d\cxle SV_d
    \quad\Longleftrightarrow\quad
    R|V_d|\icxle S|V_d|.
\]
\end{proposition}

\begin{proof}
For all \(\xi\in\sphere{d-1}\), rotational invariance and independence
give
\[
    \xi'X \stackrel{\mathcal D}{=} R \xi'U_d\stackrel{\mathcal D}{=}R e_1'U_d = RV_d
    \quad\text{and}\quad
    \xi'Y \stackrel{\mathcal D}{=} S \xi'U_d \stackrel{\mathcal D}{=}Se_1'U_d=SV_d.
\]
This proves the first equivalence. The second follows from
Lemma~\ref{lem:symmetric-cx}, since \(RV_d\) and \(SV_d\) have symmetric distributions.
\end{proof}

\begin{remark}\label{rem:angular-density}
\(V_d\) has density\footnote{Using the Gaussian representation
\(U_d\stackrel{\mathcal D}{=}Z/\|Z\|\), where
\(Z\sim\mathcal N(0,I_d)\), one has
\[
    V_d^2\stackrel{\mathcal D}{=}
    \frac{Z_1^2}{Z_1^2+\cdots+Z_d^2}.
\]
The random variables \(Z_1^2\) and \(Z_2^2+\cdots+Z_d^2\) are
independent Gamma random variables with common scale parameter
and respective shape parameters \(\frac12\) and \(\frac{d-1}{2}\).
Hence \(V_d^2\sim\operatorname{Beta}\left(\frac12,\frac{d-1}{2}\right)\).
The stated densities follow by a change of variables and the
symmetry of \(V_d\).}
\[
    v\longmapsto
    \frac{\Gamma\left(\frac d2\right)}
         {\sqrt{\pi}\,\Gamma\left(\frac{d-1}{2}\right)}
    (1-v^2)^{\frac{d-3}{2}}\mathbf 1_{\{|v|<1\}}.
\]
Consequently, \(|V_d|\) has density
\[
    f_d(v)=
    \frac{2\Gamma\left(\frac d2\right)}
         {\sqrt{\pi}\,\Gamma\left(\frac{d-1}{2}\right)}
    (1-v^2)^{\frac{d-3}{2}}\mathbf 1_{\{0<v<1\}}.
\]
\end{remark}

Propositions~\ref{prop:radial-cx} and~\ref{prop:radial-projections} reduce
the search for radial counterexamples to the one-dimensional problem of finding nonnegative random variables \(R\) and \(S\) in \(L^1\),
independent of \(|V_d|\), such that
\[
    R|V_d|\icxle S|V_d|
    \quad\text{but}\quad
    R\not\icxle S.
\]

\subsubsection{A simple counterexample with two-point radii}

We look for radii (independent of \(U_d\)) with two-point distributions of the form
\begin{equation}\label{eq:direct-radial-distributions}
    \distribution(R)=(1-p)\delta_0+p\delta_1 \quad \text{and} \quad
    \distribution(S)=(1-q)\delta_a+q\delta_b,
\end{equation}
where \(0<a<1<b\) and \(p,q\in(0,1)\). We impose
\[
    \E[S]=(1-q)a+qb<1,
\]
which is equivalent to
\[
    0<q<\frac{1-a}{b-a}.
\]
We fix \(a\), \(b\) and \(q\) satisfying these conditions and determine the ranges
of \(p\) for which multivariate convex order and one-dimensional projected convex order
hold.

\begin{proposition}[Threshold for multivariate convex order]
\label{prop:radial-threshold}
Set
\[
    p_{\mathrm{cx}}=\frac{q(b-a)}{1-a}.
\]
Then \(0<p_{\mathrm{cx}}<\E[S]<1\), and
\[
    \forall p\in(0,1),\  RU_d\cxle SU_d
    \quad\Longleftrightarrow\quad
    p\le p_{\mathrm{cx}}.
\]
\end{proposition}

\begin{proof}Since \(\E[S] < 1\), we have
\[0 < p_{\mathrm{cx}} = \frac{\E[S] - a}{1-a}< \E[S] <1.\]

By Proposition~\ref{prop:radial-cx}, it suffices to compare the radii
for increasing convex order. The stop-loss transforms are given for \(t \in \R_+\) by\footnote{We can restrict to nonnegative thresholds because the random variables are nonnegative.}
\[
    C_R(t) = \E[(R-t)_+]=p(1-t)_+
    \quad\text{and}\quad
    C_S(t) = \E[(S-t)_+]=(1-q)(a-t)_++q(b-t)_+.
\]
These expressions being piecewise affine in \(t\) and equal to \(0\) for \(t \ge b\), checking their values at \(0\), \(a\) and \(1\) gives precisely the necessary and sufficient conditions for \(R\icxle S\):
\[
    p\le\E[S]
    \quad\text{and}\quad
    p(1-a)\le q(b-a).
\]
Since \(p_{\mathrm{cx}}<\E[S]\), these conditions reduce to \(p\le \frac{q(b-a)}{1-a} = p_{\mathrm{cx}}\).
\end{proof}

For one-dimensional projected convex order, the relevant stop-loss transforms involve
\(|V_d|\), and we obtain the following characterisation.

\begin{proposition}[Threshold for one-dimensional projected convex order]
\label{prop:projected-radial-threshold}
Define
\[
    K_d:\ (t,r)\in\R_+\times\R_+
    \longmapsto \E[(r|V_d|-t)_+].
\]
The minimum
\[
    p_{\mathrm{proj}}
    =\min_{0\le t<1}
      \frac{(1-q)K_d(t,a)+qK_d(t,b)}{K_d(t,1)}
\]
exists, is attained in \((0,1)\), and satisfies
\[
    0<p_{\mathrm{proj}}<\E[S]<1.
\]
Moreover,
\[
    \forall p\in(0,1),\ 
    RU_d\projcxle SU_d
    \quad\Longleftrightarrow\quad
    p\le p_{\mathrm{proj}}.
\]
\end{proposition}

\begin{proof}
Writing \(F_d\) for the cumulative distribution function of the random variable \(|V_d|\), integration by parts gives, for \(r>0\),
\[
\begin{aligned}
    K_d(t,r)
    &=\mathbf 1_{\{t<r\}}\int_{t/r}^1(rv-t)f_d(v)\,dv\\
    &=(r-t)_+-r\int_{\min(t/r,1)}^1F_d(v)\,dv.
\end{aligned}
\]
The latter expression shows continuity for \(r>0\).
Since \(K_d(t,0)=0\) and \(0\le K_d(t,r)\le r\),
\(K_d\) is continuous on \(\R_+\times\R_+\).
Since \(f_d\) is positive on \((0,1)\), the former expression gives
\[
    \forall t\in[0,1),\  K_d(t,1)>0.
\]
We may therefore define
\[
    H_d:\ t\in[0,1)\longmapsto
    \frac{(1-q)K_d(t,a)+qK_d(t,b)}{K_d(t,1)}.
\]
This function is continuous and positive.
As \(t\uparrow1\), its numerator tends to
\(qK_d(1,b)>0\), whereas its denominator tends to zero.
Consequently,
\[
    \lim_{t\uparrow1}H_d(t)=+\infty,
\]
so \(H_d\) attains a positive minimum on \([0,1)\).

Moreover, \(H_d(0)=\E[S]\). Since \(|V_d|>0\) almost surely, dominated convergence gives
\[
    \forall r>0,\ 
    \partial_{t+}K_d(0,r)=\lim_{t\downarrow0}
      \E\left[\frac{(r|V_d|-t)_+-r|V_d|}{t}\right]=-1.
\]
Straightforward computation then gives
\[
\begin{aligned}
    H_{d,+}'(0)
    &=\frac{\E[S]-1}{\E[|V_d|]}<0.
\end{aligned}
\]
Thus \(p_{\mathrm{proj}}<\E[S]\), and every minimiser belongs
to \((0,1)\).

By Proposition~\ref{prop:radial-projections}, one-dimensional
projected convex order is equivalent to
\(R|V_d|\icxle S|V_d|\). For nonnegative random variables,
it suffices to compare stop-loss transforms at nonnegative
thresholds. Since \(K_d(t,1)\) vanishes for \(t\ge1\)
and is positive for \(t\in[0,1)\), we obtain
\[
\begin{aligned}
    RU_d\projcxle SU_d
    &\quad\Longleftrightarrow\quad
    \forall t\ge0,\ 
    pK_d(t,1)\le(1-q)K_d(t,a)+qK_d(t,b)\\
    &\quad\Longleftrightarrow\quad
    \forall t\in[0,1),\  p\le H_d(t)\\
    &\quad\Longleftrightarrow\quad
    p\le\min_{0\le t<1}H_d(t)=p_{\mathrm{proj}}.
\end{aligned}
\]
\end{proof}

The key point now is to see that the two thresholds are distinct.

\begin{proposition}[Strict gap between the thresholds]
\label{prop:radial-strict-gap}
We have
\[
    p_{\mathrm{cx}}<p_{\mathrm{proj}}.
\]
\end{proposition}

\begin{proof}
Since \(a<1<b\), we write
\[
    1=(1-\theta)a+\theta b
    \quad\text{where}\quad
    \theta=\frac{1-a}{b-a}\in(0,1).
\]
For \(t\ge0\), the function \(r\mapsto K_d(t,r)\)
is convex, so
\[
    \forall t\ge0,\ 
    K_d(t,1)\le(1-\theta)K_d(t,a)+\theta K_d(t,b).
\]
Using \(q=p_{\mathrm{cx}}\theta\), we obtain
\begin{equation}\label{eq:multiplier-jensen}
    \forall t\in[0,1),\ 
    H_d(t)-p_{\mathrm{cx}}
    =(1-p_{\mathrm{cx}})\frac{K_d(t,a)}{K_d(t,1)} +p_{\mathrm{cx}}\left(
      \frac{(1-\theta)K_d(t,a)+\theta K_d(t,b)-K_d(t,1)}
           {K_d(t,1)}\right).
\end{equation}
Both terms on the right-hand side are nonnegative.

If \(0\le t<a\), then \(K_d(t,a)>0\).

If \(a\le t<1\), positivity of \(f_d\) on \((0,1)\)
gives
\[
    \forall r>t,\ 
    \partial_r^2K_d(t,r)
    =\frac{t^2}{r^3}f_d(t/r)>0.
\]
Thus \(r\mapsto K_d(t,r)\) is strictly convex on \((t,+\infty)\).
Since \(a\le t<1<b\), equality in
\[
    K_d(t,1)\le(1-\theta)K_d(t,a)+\theta K_d(t,b)
\]
would force this convex function to be affine on \([a,b]\),
contradicting its strict convexity on \((t,b)\).
The second term in \eqref{eq:multiplier-jensen} is therefore
positive when \(a\le t<1\).

Combining the two cases, we obtain
\[
    \forall t\in[0,1),\  H_d(t)>p_{\mathrm{cx}}.
\]
Since the minimum defining \(p_{\mathrm{proj}}\) is attained,
it follows that \(p_{\mathrm{proj}}>p_{\mathrm{cx}}\).
\end{proof}

Combining these three propositions gives the desired counterexample.

\begin{theorem}[Radial counterexample in every dimension \(d\ge 2\)]
\label{thm:radial}
Fix \(d\ge2\), \(0<a<1<b\) and \(0<q<(1-a)/(b-a)\). Define
\[
    p_{\mathrm{cx}}=\frac{q(b-a)}{1-a} \quad\text{and}\quad
    p_{\mathrm{proj}}
    =\min_{0\le t<1}
       \frac{(1-q)K_d(t,a)+qK_d(t,b)}{K_d(t,1)}.
\]
Then
\[
    0<p_{\mathrm{cx}}<p_{\mathrm{proj}}<(1-q)a+qb<1.
\]
For any \(p\in(p_{\mathrm{cx}},p_{\mathrm{proj}}]\), consider random variables \(R\) and \(S\) distributed as in  \eqref{eq:direct-radial-distributions} and \(U_d\) uniformly distributed on \(\sphere{d-1}\), independent of \(R\) and \(S\).
Then
\[
    RU_d\projcxle SU_d
    \quad\text{but}\quad
    RU_d\not\cxle SU_d.
\]
\end{theorem}

\begin{proof}
The result follows from Propositions~\ref{prop:radial-threshold},
\ref{prop:projected-radial-threshold}, and~\ref{prop:radial-strict-gap}.
\end{proof}

\begin{remark}
With the above random variables, the failure of multivariate convex order can be seen explicitly by
using the convex test function \(x\mapsto(\|x\|-a)_+\), since
\[
    \E[(\|RU_d\|-a)_+]-\E[(\|SU_d\|-a)_+] = p(1-a) - q(b-a)
    =(p-p_{\mathrm{cx}})(1-a)>0.
\]
\end{remark}

\begin{remark}[A possible extension]
The specific form of \(|V_d|\) plays no role
in the preceding construction: we only use the fact that it takes
values in \([0,1]\) and has a density that is positive
on \((0,1)\). It is natural to expect that the existence of a density is not
essential. More precisely, if \(W\) is a bounded nonnegative
random variable with \(\mathbb P(W>0)>0\), and the conditional
distribution of \(W\) given \(W>0\) is not a Dirac mass, then
one should still be able to find nonnegative random variables
\(R\) and \(S\), each taking exactly two values and independent
of~\(W\), such that
\[
    RW\icxle SW
    \quad\text{but}\quad
    R\not\icxle S.
\] Establishing this extension may require choosing the parameters
\(a\), \(b\) and \(q\) to suit the distribution of~\(W\), rather than fixing
them arbitrarily subject to the conditions imposed above.
\end{remark}

\subsubsection{A radial counterexample obtained by a polynomial perturbation}
\label{sec:radial-polynomial}

While writing this paper, I sought a counterexample with absolutely
continuous radial distributions and tested generative AI tools
for this purpose. The polynomials used below were found entirely
by these tools. I reorganised the arguments in my own style,
reusing the notation introduced earlier and removing unnecessary
constants introduced by generative AI tools.

Define
\[
    Q:\ s\in[0,1]\longmapsto
    \left(s^2-\frac12\right)^2-\frac1{28},
\]
and set
\begin{equation*}
    P:\ s\in[0,1]\longmapsto
    s^2(1-s^2)Q(s)
    =\frac3{14}s^2-\frac{17}{14}s^4+2s^6-s^8.
\end{equation*}

\begin{lemma}[A probability density]
\label{lem:radial-polynomial-density}
The polynomial \(P\) satisfies
\[
    \forall s\in[0,1],\  |P'(s)|\le\frac{13}{14}.
\]
Consequently,
\[
    s \longmapsto \bigl(1-P'(s)\bigr)\mathbf 1_{\{0<s<1\}}
\]
is a probability density function, positive on \((0,1)\).
\end{lemma}

\begin{proof}
Differentiation gives
\[
    \forall s\in[0,1],\ 
    P'(s)=2s(1-2s^2)Q(s)
          +4s^3(1-s^2)\left(s^2-\frac12\right).
\]
Using \(|Q(s)|\le\frac3{14}\),
\(s^2(1-s^2)\le\frac14\) and
\(\left|s^2-\frac12\right|\le\frac12\), we obtain
\[
    \forall s\in[0,1],\ 
    |P'(s)|\le\frac37+\frac12=\frac{13}{14}.
\]
Thus \(\forall s \in (0,1), 1-P'(s) \ge\frac1{14}>0\). Moreover, since \(P(0)=P(1)=0\), we have
\[
    \int_0^1\bigl(1-P'(s)\bigr)\,ds=1-P(1)+P(0)=1.
\]
Hence \(s\longmapsto
\bigl(1-P'(s)\bigr)\mathbf 1_{\{0<s<1\}}\)
is indeed a probability density function, positive on \((0,1)\).
\end{proof}

We next recall a standard fact about uniform distributions on spheres that allows us to reduce the analysis to dimension \(d=2\).

\begin{lemma}[Reduction to dimension \(d=2\)]
\label{lem:radial-angular-factor}
For all \(d>2\),
\[
    |V_d|\stackrel{\mathcal D}{=}B_d|V_2|,
\]
where \(B_d\) is nonnegative, independent of \(V_2\), and satisfies
\(B_d^2\sim\operatorname{Beta}\left(1,\frac{d-2}{2}\right)\).
Consequently, for nonnegative random variables
\(R\) and \(S\) in \(L^1\), independent of \(V_2\) and \(V_d\),
\[
    R|V_2|\icxle S|V_2|
    \quad\Longrightarrow\quad
    R|V_d|\icxle S|V_d|.
\]
\end{lemma}

\begin{proof}
Let \(Z=(Z_1,\ldots,Z_d)'\sim\Gauss(0,I_d)\). The Gaussian
representation of uniform directions gives
\[
\begin{aligned}
    V_d^2
    \stackrel{\mathcal D}{=}
    \frac{Z_1^2}{\|Z\|^2}
    &=
    \frac{Z_1^2}{Z_1^2+Z_2^2}
    \frac{Z_1^2+Z_2^2}{\|Z\|^2}\\
    &=
    \frac{Z_1^2}{Z_1^2+Z_2^2}
    \frac{Z_1^2+Z_2^2}
         {Z_1^2+Z_2^2+(Z_3^2+\cdots+Z_d^2)}.
\end{aligned}
\]
The first factor, \(\frac{Z_1^2}{Z_1^2+Z_2^2}\), depends only on
the direction of \((Z_1,Z_2)\). By rotational invariance, this
direction is independent of the radius \(\sqrt{Z_1^2+Z_2^2}\).
Since \((Z_3,\ldots,Z_d)\) is independent of \((Z_1,Z_2)\),
the direction is in fact independent of the pair
\[
    \bigl(Z_1^2+Z_2^2,\;Z_3^2+\cdots+Z_d^2\bigr).
\]
It is therefore independent of the second factor,
\[
    \frac{Z_1^2+Z_2^2}{Z_1^2+\cdots+Z_d^2}.
\]
The two factors are thus independent. Moreover,
\[
    \frac{Z_1^2}{Z_1^2+Z_2^2}
    \stackrel{\mathcal D}{=}V_2^2
    \quad\text{and}\quad
    \frac{Z_1^2+Z_2^2}{\|Z\|^2}
    \sim\operatorname{Beta}\left(1,\frac{d-2}{2}\right),
\]
where the latter distribution follows from the fact that
\(Z_1^2+Z_2^2\) and \(Z_3^2+\cdots+Z_d^2\) are independent Gamma random variables with common scale parameter and respective shape parameters \(1\) and~\(\frac{d-2}{2}\).

Taking square roots gives
\[
    |V_d|\stackrel{\mathcal D}{=}B_d|V_2|,
\]
with \(B_d\) independent of \(V_2\) and
\(B_d^2\sim\operatorname{Beta}\left(1,\frac{d-2}{2}\right)\).

For the final implication, take \(B_d\) independent of
\((R,S,V_2)\). Increasing convex order is preserved under
multiplication by an independent bounded nonnegative random
variable, so
\[
    R|V_2|\icxle S|V_2|
    \quad\Longrightarrow\quad
    RB_d|V_2|\icxle SB_d|V_2|.
\]
The distributional identity above then yields
\(R|V_d|\icxle S|V_d|\).
\end{proof}

\begin{proposition}[An absolutely continuous radial counterexample]
\label{prop:radial-polynomial}
Let \(R\) be uniform on \([0,1]\) and let \(S\) have density
\[
    f_S(s)=\bigl(1-P'(s)\bigr)\mathbf 1_{\{0<s<1\}}.
\]
Fix \(d\ge2\) and let \(U_d\) be uniform on
\(\sphere{d-1}\), independent of \(R\) and \(S\). Then
\[ 
    RU_d\projcxle SU_d
    \quad\text{but}\quad
    RU_d\not\cxle SU_d.
\]
\end{proposition}

\begin{proof}
Lemma~\ref{lem:radial-polynomial-density} ensures that
\(f_S\) is a probability density function. We first establish
one-dimensional projected convex order in dimension \(d=2\).
Recall that \(|V_2|\) has density
\[
    f_2(v)=\frac{2}{\pi\sqrt{1-v^2}}
           \mathbf 1_{\{0<v<1\}}.
\]
Differentiating the integral expression for \(K_2\) gives
\[
    \forall t>0,\ \forall s>t,\ 
    \partial_sK_2(t,s)
    =\int_{t/s}^1v f_2(v)\,dv
    =\frac2\pi\sqrt{1-\frac{t^2}{s^2}}.
\]
Since \(K_2(t,s)=0\) for \(s\le t\), integration by parts
yields, initially for \(0<t<1\) and then by continuity
at \(0\) and \(1\),
\[
\begin{aligned}
    \forall t\in[0,1],\quad
    &\E[(S|V_2|-t)_+]-\E[(R|V_2|-t)_+]
    =\int_0^1K_2(t,s)\bigl(f_S(s)-1\bigr)\,ds\\
    &\quad=-\int_t^1K_2(t,s)P'(s)\,ds
    =\frac2\pi\int_t^1P(s)\sqrt{1-\frac{t^2}{s^2}}\,ds.
\end{aligned}
\]

Using the change of variable \(u=s^2-t^2\), we obtain
\[
\begin{aligned}
    \forall t\in[0,1],\ 
    \int_t^1P(s)\sqrt{1-\frac{t^2}{s^2}}\,ds
    &=\int_t^1s(1-s^2)
      \left(\left(s^2-\frac12\right)^2-\frac1{28}\right)
      \sqrt{s^2-t^2}\,ds\\
    &=\frac12\int_0^{1-t^2}(1-t^2-u)
      \left(\left(t^2+u-\frac12\right)^2-\frac1{28}\right)
      u^{\frac12}\,du.
\end{aligned}
\]
Expanding the polynomial and integrating each term gives
\begin{equation*}
\begin{aligned}
    \forall t\in[0,1],\ 
    &\int_t^1P(s)\sqrt{1-\frac{t^2}{s^2}}\,ds\\
    &\quad=(1-t^2)^{\frac52}
      \left(
        \frac2{15}\left(\left(t^2-\frac12\right)^2-\frac1{28}\right)
        +\frac4{35}\left(t^2-\frac12\right)(1-t^2)
        +\frac2{63}(1-t^2)^2
      \right)\\
    &\quad=\frac{16}{315}(1-t^2)^{\frac52}
      \left(t^2-\frac14\right)^2.
\end{aligned}
\end{equation*}
Consequently,
\[
\begin{aligned}
    \forall t\in[0,1],\ 
    &\E[(S|V_2|-t)_+]-\E[(R|V_2|-t)_+]\\
    &\qquad=\frac{32}{315\pi}
      (1-t^2)^{\frac52}\left(t^2-\frac14\right)^2
      \ge0.
\end{aligned}
\]
Both stop-loss transforms vanish for \(t\ge1\), so
\(R|V_2|\icxle S|V_2|\).
Lemma~\ref{lem:radial-angular-factor} and
Proposition~\ref{prop:radial-projections} therefore give
\[
    \forall d\ge2,\  RU_d\projcxle SU_d.
\]

Finally,
\[
    \forall s\in[0,1],\ 
    \mathbb P(S>s)-\mathbb P(R>s)
    =-\int_s^1P'(u)\,du=P(s).
\]
Integrating this difference of survival functions gives
\[
    \E\left[\left(S-\frac12\right)_+\right]
    -\E\left[\left(R-\frac12\right)_+\right]
    =\int_{\frac12}^1P(s)\,ds
    =\left[
        \frac1{14}s^3-\frac{17}{70}s^5
        +\frac27s^7-\frac19s^9
     \right]_{\frac12}^{1}
    =-\frac{29}{161280}<0.
\]
Thus \(R\not\icxle S\), which proves that
\(RU_d\not\cxle SU_d\) for all \(d\ge2\).
\end{proof}

\section{A Gaussian dominated by a two-component Gaussian mixture}
\label{sec:gaussian-mixtures}

The preceding section shows that one-dimensional projected convex
order does not, in general, characterise multivariate convex order,
even in the seemingly simple setting of radial distributions.
So far, the non-trivial positive results are limited to elliptical distributions
with a common radial source and the specific cases covered by
Jourdain and Pagès. We now show that the equivalence in fact holds
when comparing a centred Gaussian distribution with a mixture of
\(n=2\) centred Gaussian distributions. We begin in dimension \(d=2\),
establishing a key linear algebraic lemma and using it to derive
the probabilistic result. We then extend the argument to arbitrary dimension \(d\ge2\).

\subsection{The case \texorpdfstring{\(d=2\)}{d=2}}

The proof relies on the following key linear algebraic lemma.

\begin{lemma}[A change of basis in dimension \(d=2\)]
\label{lem:planar-cholesky}
Let \(S_1,S_2\in\mathcal S_2^{++}\) and \(p\in(0,1)\).
Let \(u\in\sphere{1}\) minimise on \(\sphere{1}\) the function
\[
    f:\ x\in\R^2\longmapsto
    p\sqrt{x'S_1x}+(1-p)\sqrt{x'S_2x}.
\]
For a unit vector \(v\perp u\), set \(O=(u\ v)\in\mathcal O_2\)
and consider the Cholesky decompositions
\[
    \forall i\in\{1,2\},\
    O'S_iO=L_iL_i'.
\]
Then
\[
    pL_1+(1-p)L_2=
    \begin{pmatrix}
        \alpha&0\\
        0&\beta
    \end{pmatrix},
\]
with \(0<\alpha=f(u)\le\beta\).
\end{lemma}

\begin{proof}
The function \(f\) is continuous on the compact set
\(\sphere{1}\), so it admits a minimiser \(u\).
Fix a unit vector \(v\perp u\).
Since \(S_1\) and \(S_2\) are positive definite, \(f\) is smooth
on \(\R^2\setminus\{0\}\). By homogeneity and Taylor's formula,
\[
\begin{aligned}
    f\left(\frac{u+tv}{\sqrt{1+t^2}}\right)
    &=\frac{f(u+tv)}{\sqrt{1+t^2}}\\
    &=f(u)+t\,v'\nabla f(u)
      +\frac{t^2}{2}\bigl(v'\nabla^2f(u)v-f(u)\bigr)
      +o(t^2).
\end{aligned}
\]
Since the left-hand side is minimised at \(t=0\), we obtain
the first-order condition
\begin{equation}
    v'\nabla f(u)
    =p\,\frac{v'S_1u}{\sqrt{u'S_1u}}
     +(1-p)\,\frac{v'S_2u}{\sqrt{u'S_2u}}
    =0
    \label{eq:two-dimensional-first-order}
\end{equation}
and the second-order condition
\begin{equation}
\begin{aligned}
    v'\nabla^2f(u)v
    &=p\left(
        \frac{v'S_1v}{\sqrt{u'S_1u}}
        -\frac{(v'S_1u)^2}{(u'S_1u)^{3/2}}
      \right)\\
    &\quad +(1-p)\left(
        \frac{v'S_2v}{\sqrt{u'S_2u}}
        -\frac{(v'S_2u)^2}{(u'S_2u)^{3/2}}
      \right)
    \ge f(u).
\end{aligned}
\label{eq:two-dimensional-second-order}
\end{equation}

Set \(O=(u\ v)\in\mathcal O_2\).
The Cholesky decompositions \(O'S_iO=L_iL_i'\) are given by
\[
    \forall i\in\{1,2\},\quad
    L_i=
    \begin{pmatrix}
        \sqrt{u'S_i u}&0\\[1ex]
        \displaystyle\frac{v'S_i u}{\sqrt{u'S_i u}}
        &
        \displaystyle\sqrt{
            v'S_i v-\frac{(v'S_i u)^2}{u'S_i u}
        }
    \end{pmatrix}.
\]
Using \eqref{eq:two-dimensional-first-order}, we have
\begin{equation}\label{eq:two-dimensional-average}
    pL_1+(1-p)L_2=
    \begin{pmatrix}
        \alpha&0\\
        0&\beta
    \end{pmatrix},
\end{equation}
where
\[
    \alpha=f(u)
    \quad\text{and}\quad
    \beta=
    p\sqrt{v'S_1v-\frac{(v'S_1u)^2}{u'S_1u}}
    +(1-p)\sqrt{v'S_2v-\frac{(v'S_2u)^2}{u'S_2u}}.
\]
Positive definiteness gives \(\alpha,\beta>0\).

We now prove that \(\beta\ge\alpha\).

First observe that
\[
    \forall i\in\{1,2\},\ \forall x\in\R^2,\quad
    \|L_i'x\|
    =\sqrt{x'L_iL_i'x}
    =\sqrt{(Ox)'S_i(Ox)}.
\]
Since \(O\) is orthogonal,
\[
    \min_{x\in\sphere{1}}
    p\|L_1'x\|+(1-p)\|L_2'x\|
    =
    \min_{x\in\sphere{1}}f(Ox)
    =
    \min_{x\in\sphere{1}}f(x)
    =\alpha.
\]
By homogeneity, it follows that
\begin{equation}\label{eq:two-dimensional-global-minimum}
    \forall x\in\R^2,\quad
    p\|L_1'x\|+(1-p)\|L_2'x\|
    \ge\alpha\|x\|.
\end{equation}

We seek a nonzero vector \(x\in\R^2\) for which equality holds
in the triangle inequality
\[
    \left\|\bigl(pL_1'+(1-p)L_2'\bigr)x\right\|
    \le p\|L_1'x\|+(1-p)\|L_2'x\|.
\]
Equality holds whenever \(L_1'x\) and \(L_2'x\) are positively
collinear. This occurs when \(x\) is an eigenvector of
\((L_2')^{-1}L_1'\) associated with a positive eigenvalue. This upper triangular matrix has eigenvalues
\[
    \sqrt{\frac{u'S_1u}{u'S_2u}}
    \quad\text{and}\quad
    \sqrt{
        \frac{
            v'S_1v-\dfrac{(v'S_1u)^2}{u'S_1u}
        }{
            v'S_2v-\dfrac{(v'S_2u)^2}{u'S_2u}
        }
    },
\]
both positive.

Suppose first that these eigenvalues are distinct.
Let \(x=(x_1,x_2)'\) be an eigenvector associated with
the second eigenvalue. The vectors \(L_1'x\) and \(L_2'x\)
are positively collinear, so equality holds in the triangle
inequality. Together with
\eqref{eq:two-dimensional-global-minimum}, this gives
\[
    \alpha\|x\|
    \le \left\|\bigl(pL_1'+(1-p)L_2'\bigr)x\right\|.
\]
Using \eqref{eq:two-dimensional-average} and squaring,
we obtain
\[
    \alpha^2(x_1^2+x_2^2)
    \le\alpha^2x_1^2+\beta^2x_2^2.
\]
Since \(e_1\) is an eigenvector associated with the first
eigenvalue and eigenvectors associated with distinct eigenvalues
are linearly independent, we have \(x_2\ne0\).
This implies \(\alpha^2\le\beta^2\), hence
\(\beta\ge\alpha\) since \(\alpha\) and \(\beta\) are positive.

Suppose now that the eigenvalues coincide. Then
\[
    \frac{v'S_1v}{u'S_1u}
    -\frac{(v'S_1u)^2}{(u'S_1u)^2}
    =
    \frac{v'S_2v}{u'S_2u}
    -\frac{(v'S_2u)^2}{(u'S_2u)^2}.
\]
Factoring out the square root of this common value in the
expression for \(\beta\), we obtain
\[
\begin{aligned}
    \beta
    &=
    \sqrt{
        \frac{v'S_1v}{u'S_1u}
        -\frac{(v'S_1u)^2}{(u'S_1u)^2}
    }\,
    \bigl(p\sqrt{u'S_1u}+(1-p)\sqrt{u'S_2u}\bigr)\\
    &=
    \alpha\sqrt{
        \frac{v'S_1v}{u'S_1u}
        -\frac{(v'S_1u)^2}{(u'S_1u)^2}
    }.
\end{aligned}
\]
Hence, for \(i\in\{1,2\}\),
\[
    \frac{v'S_iv}{u'S_iu}
    -\frac{(v'S_iu)^2}{(u'S_iu)^2}
    =\left(\frac{\beta}{\alpha}\right)^2,
\]
and therefore
\[
\begin{aligned}
    \frac{v'S_iv}{\sqrt{u'S_iu}}
    -\frac{(v'S_iu)^2}{(u'S_iu)^{3/2}}
    &=
    \sqrt{u'S_iu}\left(
        \frac{v'S_iv}{u'S_iu}
        -\frac{(v'S_iu)^2}{(u'S_iu)^2}
    \right)\\
    &=\left(\frac{\beta}{\alpha}\right)^2\sqrt{u'S_iu}.
\end{aligned}
\]
Since \(f(u)=\alpha\), the second-order condition
\eqref{eq:two-dimensional-second-order} thus becomes
\[
    \left(\frac{\beta}{\alpha}\right)^2
    \bigl(p\sqrt{u'S_1u}+(1-p)\sqrt{u'S_2u}\bigr)
    =\frac{\beta^2}{\alpha}
    \ge\alpha.
\]
Thus \(\beta^2\ge\alpha^2\), and positivity of \(\alpha\) and
\(\beta\) again yields \(\beta\ge\alpha\).
\end{proof}

We now show that Lemma~\ref{lem:planar-cholesky} yields the equivalence of
the two convex orders in dimension \(d=2\) when comparing a standard Gaussian
distribution with a mixture of two centred Gaussian
distributions with positive definite variance-covariance
matrices.

\begin{proposition}[The normalised positive definite case]
\label{prop:dimension-two-normalised}
Let \(\Sigma_1,\Sigma_2\in\mathcal S_2^{++}\) and \(p\in(0,1)\).
If
\[
    X\sim\Gauss(0,I_2)
    \quad\text{and}\quad
    Y\sim p\Gauss(0,\Sigma_1)+(1-p)\Gauss(0,\Sigma_2),
\]
then
\[
    X\projcxle Y
    \quad\Longleftrightarrow\quad
    X\cxle Y.
\]
\end{proposition}

\begin{proof}
Multivariate convex order implies one-dimensional projected
convex order, so only the converse requires proof.

Suppose that \(X\projcxle Y\). By \eqref{eq:necessary},
\[
    \forall x\in\sphere{1},\
    p\sqrt{x'\Sigma_1x}+(1-p)\sqrt{x'\Sigma_2x}\ge1.
\]
Lemma~\ref{lem:planar-cholesky} provides \(O\in\mathcal O_2\) such that the Cholesky
decompositions
\[
    \forall i\in\{1,2\},\
    O'\Sigma_iO=L_iL_i'
\]
satisfy
\[
    pL_1+(1-p)L_2=
    \begin{pmatrix}
        \alpha&0\\
        0&\beta
    \end{pmatrix}
    \quad\text{with}\quad 1\le\alpha\le\beta.
\]
Set
\[
    \forall i\in\{1,2\},\
    M_i=OL_i,
\]
and write \(M=pM_1+(1-p)M_2\). Then
\[
    \forall i\in\{1,2\},\
    M_iM_i'=\Sigma_i,
\]
and
\[
    MM'=O
    \begin{pmatrix}
        \alpha^2&0\\
        0&\beta^2
    \end{pmatrix}
    O'\ge I_2.
\]
For \(Z\sim\Gauss(0,I_2)\),
\eqref{eq:gaussian-covariance} gives \(Z\cxle MZ\).
Hence, for every convex function \(\phi:\R^2\to\R\),
\[
\begin{aligned}
    \E[\phi(X)]
    &=\E[\phi(Z)]\\
    &\le\E[\phi(MZ)]\\
    &\le p\E[\phi(M_1Z)]
         +(1-p)\E[\phi(M_2Z)]\\
    &=\E[\phi(Y)].
\end{aligned}
\]
Thus \(X\cxle Y\).
\end{proof}

To extend this result to positive semidefinite
variance-covariance matrices, we use two auxiliary lemmas.
The first preserves the directional inequalities under
regularisation. The second allows us to pass to the limit
using Lipschitz convex test functions.

\begin{lemma}[Preservation under regularisation]
\label{lem:square-root-regularisation}
Let \(a,b,c\ge0\) and \(p\in[0,1]\). If
\[
    \sqrt a\le p\sqrt b+(1-p)\sqrt c,
\]
then, for all \(\eta\ge0\),
\[
    \sqrt{a+\eta}
    \le p\sqrt{b+\eta}+(1-p)\sqrt{c+\eta}.
\]
\end{lemma}

\begin{proof}
The assumed inequality and the triangle inequality in \(\R^2\)
give
\[
    \sqrt{a+\eta}
    \le
    \left\|
        p\begin{pmatrix}\sqrt b\\\sqrt\eta\end{pmatrix}
        +(1-p)\begin{pmatrix}\sqrt c\\\sqrt\eta\end{pmatrix}
    \right\|
    \le p\sqrt{b+\eta}+(1-p)\sqrt{c+\eta}.
\]
\end{proof}

\begin{lemma}[Lipschitz convex test functions]
\label{lem:lipschitz-tests}
Let \(X\) and \(Y\) be \(\R^d\)-valued random variables in \(L^1\).
Then
\begin{equation*}
\begin{aligned}
    X\cxle Y\quad\Longleftrightarrow\quad
    &\E[\phi(X)]\le\E[\phi(Y)]\\
    &\text{for all Lipschitz convex functions }\phi:\R^d\to\R.
\end{aligned}
\end{equation*}
\end{lemma}

\begin{proof}
The forward implication follows directly from the definition
of convex order. Conversely, suppose that the comparison
holds for every Lipschitz convex function.
Affine tests give \(\E[X]=\E[Y]\).

Let \(\phi:\R^d\to\R\) be convex. Choose an affine supporting
function \(\ell\) for \(\phi\) at \(0\) and set
\(\widetilde\phi=\phi-\ell\), so that
\(\widetilde\phi\ge0\) and \(\widetilde\phi(0)=0\).

For all integers \(m\ge1\), define the inf-convolution
\[
    \widetilde\phi_m(x)
    =\inf_{y\in\R^d}
      \bigl\{\widetilde\phi(y)+m\|x-y\|\bigr\}.
\]
The function
\((x,y)\mapsto\widetilde\phi(y)+m\|x-y\|\)
is convex, so its partial infimum
\(\widetilde\phi_m\) is convex.
The triangle inequality gives
\[
    \forall x,z\in\R^d,\
    |\widetilde\phi_m(x)-\widetilde\phi_m(z)|
    \le m\|x-z\|,
\]
hence \(\widetilde\phi_m\) is \(m\)-Lipschitz.
Moreover,
\[
    \forall x\in\R^d,\
    0\le\widetilde\phi_m(x)
    \le\min\{\widetilde\phi(x),m\|x\|\},
\]
and \(\widetilde\phi_m(x)\) is nondecreasing in \(m\).

We claim that
\(\widetilde\phi_m(x)\uparrow\widetilde\phi(x)\)
for all \(x\in\R^d\).
Fix \(x\) and choose approximate minimisers \(y_m\) such that
\[
    \widetilde\phi(y_m)+m\|x-y_m\|
    \le\widetilde\phi_m(x)+\frac1m
    \le\widetilde\phi(x)+\frac1m.
\]
Since \(\widetilde\phi\ge0\), this implies \(y_m\underset{m\to+\infty}{\longrightarrow} x\).
Also,
\(\widetilde\phi_m(x)\ge\widetilde\phi(y_m)-1/m\),
so continuity of \(\widetilde\phi\) gives
\[
    \liminf_{m\to+\infty}\widetilde\phi_m(x)
    \ge\lim_{m\to+\infty}\widetilde\phi(y_m)
    =\widetilde\phi(x).
\]
Together with
\(\widetilde\phi_m(x)\le\widetilde\phi(x)\),
this proves the claim.

The assumed comparison applies to each
\(\widetilde\phi_m\), giving
\[
    \forall m\ge1,\
    \E[\widetilde\phi_m(X)]
    \le\E[\widetilde\phi_m(Y)].
\]
Monotone convergence yields
\(\E[\widetilde\phi(X)]\le\E[\widetilde\phi(Y)]\).
Since \(\ell\) is affine and \(\E[X]=\E[Y]\),
adding the equal finite expectations
\(\E[\ell(X)]\) and \(\E[\ell(Y)]\) gives
\(\E[\phi(X)]\le\E[\phi(Y)]\).

Thus \(X\cxle Y\).
\end{proof}

We can now state the result in dimension \(d=2\) without any
invertibility assumption.

\begin{theorem}[Comparison with a two-component Gaussian mixture in dimension \(d=2\)]
\label{thm:dimension-two}
Let \(\Sigma,\Sigma_1,\Sigma_2\in\mathcal S_2^+\)
and \(p\in(0,1)\). If
\[
    X\sim\Gauss(0,\Sigma)
    \quad\text{and}\quad
    Y\sim p\Gauss(0,\Sigma_1)+(1-p)\Gauss(0,\Sigma_2),
\]
then
\[
    X\projcxle Y
    \quad\Longleftrightarrow\quad
    X\cxle Y.
\]
\end{theorem}

\begin{proof}
Multivariate convex order implies one-dimensional projected
convex order, so only the converse requires proof.
Suppose that \(X\projcxle Y\).

First assume that all three matrices are positive definite.
Both orders are preserved under linear transformations, so
\[
    \Sigma^{-1/2}X\projcxle\Sigma^{-1/2}Y.
\]
The random variable on the left is standard Gaussian.
The mixture on the right has positive definite
variance-covariance matrices
\[
    \forall i\in\{1,2\},\
    \widetilde\Sigma_i
    =\Sigma^{-1/2}\Sigma_i\Sigma^{-1/2}.
\]
Proposition~\ref{prop:dimension-two-normalised} gives therefore
\[
    \Sigma^{-1/2}X\cxle\Sigma^{-1/2}Y.
\]
Applying \(\Sigma^{1/2}\) yields \(X\cxle Y\).

We now consider the general positive semidefinite case.
By \eqref{eq:necessary},
\[
    \forall x\in\R^2,\
    \sqrt{x'\Sigma x}
    \le p\sqrt{x'\Sigma_1x}
         +(1-p)\sqrt{x'\Sigma_2x}.
\]
For \(\varepsilon>0\), apply
Lemma~\ref{lem:square-root-regularisation} with
\(\eta=\varepsilon\|x\|^2\) to obtain
\[
    \forall x\in\R^2,\
    \sqrt{x'(\Sigma+\varepsilon I_2)x}
    \le p\sqrt{x'(\Sigma_1+\varepsilon I_2)x}
         +(1-p)\sqrt{x'(\Sigma_2+\varepsilon I_2)x}.
\]
Let \(Z\sim\Gauss(0,I_2)\) be independent of \((X,Y)\).
By \eqref{eq:necessary},
\[
    X+\sqrt{\varepsilon}\,Z
    \projcxle Y+\sqrt{\varepsilon}\,Z.
\]
The positive definite case therefore gives
\[
    X+\sqrt{\varepsilon}\,Z
    \cxle Y+\sqrt{\varepsilon}\,Z.
\]
For all Lipschitz convex functions \(\phi:\R^2\to\R\)
with Lipschitz constant \(L\), conditional Jensen's
inequality and the preceding comparison yield
\[
\begin{aligned}
    \E[\phi(X)]
    &\le\E[\phi(X+\sqrt{\varepsilon}\,Z)]\\
    &\le\E[\phi(Y+\sqrt{\varepsilon}\,Z)]\\
    &\le\E[\phi(Y)]
         +L\sqrt{\varepsilon}\,\E[\|Z\|].
\end{aligned}
\]
Letting \(\varepsilon\downarrow0\) and applying
Lemma~\ref{lem:lipschitz-tests} gives \(X\cxle Y\).
\end{proof}

\subsection{Equivalence in arbitrary dimension}

We now extend the change of basis argument to arbitrary dimension by
induction.\footnote{Benjamin Jourdain suggested extending the
two-dimensional construction by induction when I presented
the argument establishing the result in dimension \(d=2\) to him on the sidelines of the London-Paris Bachelier Workshop in September 2026. The proof below follows his
suggestion, and the credit for the idea of the method used
for the extension to arbitrary dimension should go mainly
to him.} After choosing a minimising direction, we apply the induction
hypothesis to the two Schur complements on its orthogonal
complement. The two-dimensional lemma is key to ensuring that the
smallest diagonal coefficient obtained in this way is at
least the initial minimum.

\begin{proposition}[A change of basis in arbitrary dimension]
\label{prop:matrix-construction}
Let \(d\ge2\), \(S_1,S_2\in\mathcal S_d^{++}\) and \(p\in(0,1)\). There exists \(O\in\mathcal O_d\) such that the Cholesky decompositions
\[
    \forall i\in\{1,2\},\quad O'S_iO=L_iL_i'
\]
satisfy
\[
    pL_1+(1-p)L_2=\begin{pmatrix}\eta_1&&0\\&\ddots&\\0&&\eta_d\end{pmatrix} \quad \text{with} \quad 0<\eta_1\le\cdots\le\eta_d,
\]
where
\[
    \eta_1=\min_{x\in\sphere{d-1}} p\sqrt{x'S_1x}+(1-p)\sqrt{x'S_2x}.
\]
\end{proposition}

\begin{proof}
We argue by induction on \(d\). The case \(d=2\) is Lemma~\ref{lem:planar-cholesky}. Let \(d\ge3\) and suppose that the result holds in dimension \(d-1\).

Choose a minimiser \(u\in\sphere{d-1}\) of
\[
    f:x\in\sphere{d-1}\longmapsto p\sqrt{x'S_1x}+(1-p)\sqrt{x'S_2x},
\]
and set \(\eta_1=f(u)>0\). Such a minimiser exists by continuity and compactness.

Let \(V\in\mathcal M_{d,d-1}\) have columns forming an orthonormal basis of \(\operatorname{span}(u)^\perp\), so that \((u\ V)\in\mathcal O_d\). The first-order condition at \(u\) then reads
\begin{equation}\label{eq:construction-first-order}
    p\,\frac{V'S_1u}{\sqrt{u'S_1u}}+(1-p)\,\frac{V'S_2u}{\sqrt{u'S_2u}}=0.
\end{equation}

For each \(i\in\{1,2\}\), the matrix \(S_i\) is positive definite, hence so is \((u\ V)'S_i(u\ V)\). Its Schur complement
\begin{equation}\label{eq:construction-schur}
    \widetilde S_i=V'S_iV-\frac{V'S_iu\,u'S_iV}{u'S_iu}
\end{equation}
therefore belongs to \(\mathcal S_{d-1}^{++}\).

We first show that
\begin{equation}\label{eq:construction-minimum-bound}
    \forall y\in\sphere{d-2},\quad p\sqrt{y'\widetilde S_1y}+(1-p)\sqrt{y'\widetilde S_2y}\ge\eta_1.
\end{equation}
Fix \(y\in\sphere{d-2}\) and set \(v=Vy\). Then \(u\) and \(v\) are orthonormal, and \(u\) minimises \(f\) on the unit circle of \(\operatorname{span}(u,v)\). For each \(i\in\{1,2\}\), the restriction of the quadratic form \(x\mapsto x'S_ix\) to this plane has matrix
\[
    \begin{pmatrix}u'S_iu&u'S_iv\\v'S_iu&v'S_iv\end{pmatrix}
\]
in the basis \((u,v)\). The second diagonal entry in the lower triangular matrix of its Cholesky decomposition is
\[
    \sqrt{v'S_iv-\frac{(v'S_iu)^2}{u'S_iu}}=\sqrt{y'\widetilde S_iy}.
\]
The two-dimensional argument of Lemma~\ref{lem:planar-cholesky}, applied in this basis with the minimiser \(u\), therefore gives
\[
    p\sqrt{y'\widetilde S_1y}+(1-p)\sqrt{y'\widetilde S_2y}\ge p\sqrt{u'S_1u}+(1-p)\sqrt{u'S_2u}=\eta_1.
\]
This proves \eqref{eq:construction-minimum-bound}.

Applying the induction hypothesis to \(\widetilde S_1,\widetilde S_2\), we obtain \(\widetilde O\in\mathcal O_{d-1}\) such that the Cholesky decompositions
\[
    \forall i\in\{1,2\},\quad \widetilde O'\widetilde S_i\widetilde O=\widetilde L_i\widetilde L_i'
\]
satisfy
\[
    p\widetilde L_1+(1-p)\widetilde L_2=\begin{pmatrix}\eta_2&&0\\&\ddots&\\0&&\eta_d\end{pmatrix} \quad \text{with} \quad 0<\eta_2\le\cdots\le\eta_d,
\]
where
\[
    \eta_2=\min_{y\in\sphere{d-2}} p\sqrt{y'\widetilde S_1y}+(1-p)\sqrt{y'\widetilde S_2y}\ge\eta_1.
\]
Now, set \(O=(u\ V\widetilde O)\in\mathcal O_d\) and
\[
    \forall i\in\{1,2\},\quad L_i=\begin{pmatrix}\sqrt{u'S_iu}&0\\[1ex]\dfrac{\widetilde O'V'S_iu}{\sqrt{u'S_iu}}&\widetilde L_i\end{pmatrix}.
\]
By \eqref{eq:construction-schur},
\begin{equation}\label{eq:construction-factorisation}
    \forall i\in\{1,2\},\quad L_iL_i'=O'S_iO.
\end{equation}
Since each \(L_i\) is lower triangular with positive diagonal entries, the identities in \eqref{eq:construction-factorisation} are Cholesky decompositions.

Finally, \eqref{eq:construction-first-order} gives
\[
    pL_1+(1-p)L_2=\begin{pmatrix}\eta_1&0\\0&p\widetilde L_1+(1-p)\widetilde L_2\end{pmatrix}=\begin{pmatrix}\eta_1&&0\\&\ddots&\\0&&\eta_d\end{pmatrix},
\]
with \(0<\eta_1\le\cdots\le\eta_d\). This completes the induction.
\end{proof}

We now apply the above construction to prove the equivalence
between one-dimensional projected convex order and multivariate
convex order for a centred Gaussian distribution dominated by
a mixture of two centred Gaussian distributions.

\begin{theorem}[Two-component Gaussian mixtures]
\label{thm:two-components-arbitrary-dimension}
Let \(d\ge2\), \(p\in(0,1)\) and
\(\Sigma,\Sigma_1,\Sigma_2\in\mathcal S_d^+\). Let
\[
    X\sim\Gauss(0,\Sigma)
    \quad\text{and}\quad
    Y\sim p\Gauss(0,\Sigma_1)+(1-p)\Gauss(0,\Sigma_2).
\]
Then
\[
    X\cxle Y\quad\Longleftrightarrow\quad X\projcxle Y.
\]
\end{theorem}

\begin{proof}
Only the reverse implication needs to be proved.

Suppose that \(X\projcxle Y\). By \eqref{eq:necessary},
\begin{equation}\label{eq:two-components-directional}
    \forall \xi\in\R^d,\quad \sqrt{\xi'\Sigma\xi}\le p\sqrt{\xi'\Sigma_1\xi}+(1-p)\sqrt{\xi'\Sigma_2\xi}.
\end{equation}
First assume that \(\Sigma=I_d\) and \(\Sigma_1,\Sigma_2\in\mathcal S_d^{++}\). Apply Proposition~\ref{prop:matrix-construction} with \(S_i=\Sigma_i\), for \(i\in\{1,2\}\), and set \(M_i=OL_i\). Then \(M_iM_i'=\Sigma_i\), and \eqref{eq:two-components-directional} gives
\[
    \eta_1=\min_{x\in\sphere{d-1}} p\sqrt{x'\Sigma_1x}+(1-p)\sqrt{x'\Sigma_2x}\ge1.
\]
Since \(\eta_d\ge\cdots\ge\eta_1\ge1\), we obtain
\[
    \bigl(pM_1+(1-p)M_2\bigr)\bigl(pM_1+(1-p)M_2\bigr)'=O\begin{pmatrix}\eta_1^2&&0\\&\ddots&\\0&&\eta_d^2\end{pmatrix}O'\ge I_d.
\]
For \(Z\sim\Gauss(0,I_d)\), Gaussian comparison and Jensen's inequality therefore yield
\[
    X\cxle\bigl(pM_1+(1-p)M_2\bigr)Z\cxle Y,
\]
exactly as in the proof of Theorem~\ref{thm:dimension-two}.

If all three variance-covariance matrices are positive definite, the same change of variables as in the proof of Theorem~\ref{thm:dimension-two} reduces the comparison to the normalised case: apply the preceding argument to \(\Sigma^{-1/2}X\) and \(\Sigma^{-1/2}Y\), then use preservation of convex order under the linear map \(x\mapsto\Sigma^{1/2}x\).

Finally, if the variance-covariance matrices are only positive semidefinite, Lemma~\ref{lem:square-root-regularisation} ensures that \eqref{eq:two-components-directional} remains valid after adding \(\varepsilon I_d\) to each variance-covariance matrix, for every \(\varepsilon>0\). The positive definite case thus gives
\[
    X+\sqrt{\varepsilon}\,Z\cxle Y+\sqrt{\varepsilon}\,Z,
\]
where \(Z\sim\Gauss(0,I_d)\) is independent of \((X,Y)\). As in the proof of Theorem~\ref{thm:dimension-two}, letting \(\varepsilon\downarrow0\) in expectations of Lipschitz convex test functions and applying Lemma~\ref{lem:lipschitz-tests} yields \(X\cxle Y\).
\end{proof}

\begin{remark}[The sufficient condition of Jourdain and Pag\`es]
\label{rem:two-component-matrix-condition}
Suppose that the equivalent conditions of Theorem~\ref{thm:two-components-arbitrary-dimension} hold. When all three variance-covariance matrices are positive definite, the preceding proof constructs matrices \(M_1,M_2\in\mathcal M_d\) satisfying
\begin{equation}\label{eq:two-component-matrix-condition}
    \forall i\in\{1,2\},\quad M_iM_i'=\Sigma_i \qquad\text{and}\qquad \Sigma\le\bigl(pM_1+(1-p)M_2\bigr)\bigl(pM_1+(1-p)M_2\bigr)'.
\end{equation}
For general positive definite \(\Sigma\), these matrices are obtained by multiplying those constructed for the normalised problem with \(\Sigma=I_d\) on the left by \(\Sigma^{1/2}\). To recover the formulation \eqref{eq:sufficient}, set \(P=(pI_d\ \ (1-p)I_d)\) and
\[
    \Gamma=\begin{pmatrix}\Sigma_1&M_1M_2'\\M_2M_1'&\Sigma_2\end{pmatrix}=\begin{pmatrix}M_1\\M_2\end{pmatrix}\begin{pmatrix}M_1'&M_2'\end{pmatrix}\in\mathcal S_{2d}^+.
\]
Then \(\Gamma\) is the variance-covariance matrix of \((M_1Z,M_2Z)\), where \(Z\sim\Gauss(0,I_d)\), and \eqref{eq:two-component-matrix-condition} gives \(\Sigma\le P\Gamma P'\).

The same matrix representation holds when the variance-covariance matrices are singular. Indeed, in the positive definite case, polar decomposition gives \(M_i=\Sigma_i^{1/2}O_i\), with \(O_i\in\mathcal O_d\), for~\(i\in\{1,2\}\). Multiplying both matrices on the right by \(O_1'\) preserves \eqref{eq:two-component-matrix-condition} and allows us to take \(M_1=\Sigma_1^{1/2}\) and \(M_2=\Sigma_2^{1/2}O\), where \(O=O_2O_1'\in\mathcal O_d\).

For every \(\varepsilon>0\), Lemma~\ref{lem:square-root-regularisation} ensures that the directional inequalities hold for \(\Sigma+\varepsilon I_d\), \(\Sigma_1+\varepsilon I_d\) and \(\Sigma_2+\varepsilon I_d\). Applying the positive definite construction and the preceding observation yields \(O_\varepsilon\in\mathcal O_d\) such that, with
\[
    M_{1,\varepsilon}=(\Sigma_1+\varepsilon I_d)^{1/2}, \qquad M_{2,\varepsilon}=(\Sigma_2+\varepsilon I_d)^{1/2}O_\varepsilon,
\]
we have
\[
    \Sigma+\varepsilon I_d\le\bigl(pM_{1,\varepsilon}+(1-p)M_{2,\varepsilon}\bigr)\bigl(pM_{1,\varepsilon}+(1-p)M_{2,\varepsilon}\bigr)'.
\]
By compactness of \(\mathcal O_d\), there exist a sequence \(\varepsilon_m\downarrow0\) and \(\widehat O\in\mathcal O_d\) such that \(O_{\varepsilon_m}\to\widehat O\). Continuity of the positive semidefinite square root gives
\[
    M_{1,\varepsilon_m}\longrightarrow M_1:=\Sigma_1^{1/2}, \qquad M_{2,\varepsilon_m}\longrightarrow M_2:=\Sigma_2^{1/2}\widehat O.
\]
Since the positive semidefinite cone is closed, passing to the limit in the preceding inequality yields
\[
    \Sigma\le\bigl(pM_1+(1-p)M_2\bigr)\bigl(pM_1+(1-p)M_2\bigr)'.
\]
Moreover, \(M_iM_i'=\Sigma_i\) for all \(i\in\{1,2\}\), so \eqref{eq:two-component-matrix-condition} holds. Defining \(\Gamma\) by the same block formula as before therefore gives a positive semidefinite variance-covariance matrix with diagonal blocks \(\Sigma_1,\Sigma_2\) and \(\Sigma\le P\Gamma P'\). Thus, for two components, the sufficient condition \eqref{eq:sufficient} of Jourdain and Pag\`es is also necessary, including for singular variance-covariance matrices.
\end{remark}

\begin{remark}[A common radial generator]
The Gaussian assumption is not essential to the argument: the matrix construction depends only on the directional inequalities, while the probabilistic comparison uses properties shared by linear images of a common radial random vector. Therefore, the equivalence between multivariate convex order, one-dimensional projected convex order and the directional inequalities \eqref{eq:two-components-directional} remains valid for
\[
    X\stackrel{\mathcal D}{=}\Sigma^{1/2}Z, \qquad \distribution(Y)=p\distribution(\Sigma_1^{1/2}Z)+(1-p)\distribution(\Sigma_2^{1/2}Z),
\]
where \(Z\) is an \(\R^d\)-valued random variable in \(L^1\) with a radial distribution and \(\mathbb P(Z\ne0)>0\). Here, \(\Sigma,\Sigma_1,\Sigma_2\in\mathcal S_d^+\) are scatter matrices, rather than necessarily variance-covariance matrices.

Indeed, rotational invariance gives
\[
    \forall \xi\in\R^d,\ \forall A\in\mathcal M_d,\quad \E[|\xi'AZ|]=\sqrt{\xi'AA'\xi}\,\E[|Z_1|],
\]
where \(Z_1\) denotes the first coordinate of \(Z\). Since \(0<\E[|Z_1|]<+\infty\), absolute-value tests yield the directional inequalities from projected convex order. Conversely, these inequalities give the matrix representation \eqref{eq:two-component-matrix-condition}, by the construction in Remark~\ref{rem:two-component-matrix-condition}. Comparison of linear images of a common radial random vector and Jensen's inequality then yield
\[
    \Sigma^{1/2}Z\cxle\bigl(pM_1+(1-p)M_2\bigr)Z\cxle Y,
\]
where the last comparison uses \(M_iZ\stackrel{\mathcal D}{=}\Sigma_i^{1/2}Z\) for all \(i\in\{1,2\}\). This includes centred multivariate Student distributions with a common number \(\nu>1\) of degrees of freedom.
\end{remark}

\section{Concluding remarks}
\label{sec:conclusion}

One-dimensional projected convex order does not, in general,
characterise multivariate convex order, even for symmetric
or radial distributions. The radial setting makes the
distinction particularly transparent: multivariate convex
order compares the radii in increasing convex order, whereas
projected convex order compares them after multiplication
by a common independent angular factor. Our counterexamples
show that this multiplication can conceal a failure of
increasing convex order between the radii.

Nevertheless, we have proved that one-dimensional projected
convex order characterises multivariate convex order when
a centred Gaussian distribution is compared with a mixture
of two centred Gaussian distributions in every dimension \(d\ge2\).
In this setting, the sufficient condition of Jourdain and
Pag\`es is also necessary. 

The case of mixtures with \(n\ge3\) centred Gaussian
components remains open. Our proof in arbitrary dimension
relies on the two-dimensional lemma at each induction step.
For two components, equality in the triangle inequality
reduces to an eigenvector problem for the single matrix
\((L_2')^{-1}L_1'\). With more components, the analogous
argument would require an additional eigenvector
common to several matrices, beyond the trivial one
they already share. Such a common eigenvector need not
exist, which appears to prevent a direct extension of
this argument.

Future work should determine whether the equivalence holds
for mixtures with \(n\ge3\) components and how the answer
depends on the dimension \(d\). We do not propose a
conjecture on these questions.

\section*{Statements}

\subsection*{Acknowledgements}

This research was inspired by a talk given by Benjamin Jourdain at the
Séminaire Bachelier in 2025. My initial goal was simply to construct,
numerically and quickly, counterexamples to the equivalence left open
in~\cite{JourdainPages2025}. This proved somewhat presumptuous -- in the
two-component case, there are none. A touch of academic pride in the face
of this difficulty nevertheless prompted me to pursue the question further,
ultimately leading to the results presented here.

I would like to thank Benjamin Jourdain for that initial talk and for our subsequent discussions, particularly during my visit to CERMICS (École nationale des ponts et chaussées), where I presented some early and inconclusive findings, and at the London-Paris Bachelier Workshop at the Institut Henri Poincaré, where I showed him the change-of-basis trick that had led me to establish the equivalence in the two-component case in dimension \(d=2\). His suggestion to extend the construction by induction led to the argument in arbitrary dimension developed here.

I would also like to thank Guillaume Carlier for a helpful discussion over lunch and afterwards
while I was exploring the Gaussian mixture problem from a convex-geometric angle using Minkowski sums of ellipses.

Warm thanks also go to the participants in the Groupe de Travail
``Mathématiques financières et actuarielles, probabilités numériques'' at the Laboratoire de Probabilités, Statistique et Modélisation, in particular Stéphane Crépey, Idris Kharroubi and Gilles Pagès.

Finally, my thanks go to my long-standing supporter, Jean-Michel
Lasry, who has shared with me over the years his deep understanding of
convex analysis, from original applications of convex duality to inf-sup
convolution, raising me beyond the level of a \textquotedblleft Super Taupin\textquotedblright{} -- as
Pierre-Louis Lions once called me, in a sense he meant critically but
which I took as a compliment.

\subsection*{Data availability}

No datasets were generated or analysed in this study.

\subsection*{Declaration of generative AI use}

During the preparation of this manuscript, I used various versions of Google's Gemini, OpenAI's ChatGPT, and Anthropic's Claude. With the exception noted below, the mathematical ideas, results, and proof strategies were developed without the use of generative AI: I believe I will always find a particular pleasure in doing mathematics in the classical way. These tools were used mainly at the writing stage, to help reorganise proofs when needed, improve exposition and notation, identify points requiring further justification, and assist with editing, spell-checking, and syntactic refinement. As a native French speaker, I hoped they might bring my English a little closer to that of His Majesty's subjects. I also used these tools to modernise and tidy the \LaTeX{} source code.

There is one substantive exception to the above. The polynomial counterexample in Section~\ref{sec:radial-polynomial} was found using generative AI tools. Their ability to suggest counterexamples at times seems almost superhuman, even to someone who kept Bertrand Hauchecorne's well-known French book on counterexamples by his bedside during his undergraduate years -- and not merely for its fancy cover.

All AI-assisted material was reviewed, checked, and edited by me, perhaps reintroducing a few Gallicisms and misspellings along the way. I take full responsibility for the mathematical arguments and the final text.

\end{document}